\documentclass[11pt]{article}

\usepackage{amsfonts,amsmath,latexsym,color,epsfig,hyperref}
\newtheorem{theorem}{Theorem}[section]
\newtheorem{lemma}[theorem]{Lemma}
\newtheorem{proposition}[theorem]{Proposition}
\newtheorem{conjecture}{Conjecture}[section]
\newtheorem{corollary}[theorem]{Corollary}
\newtheorem{question}{Question}[section]
\newtheorem{claim}{Claim}[section]

\newenvironment{proof}
      {\medskip\noindent{\bf Proof:}\hspace{1mm}}
      {\hfill$\Box$\medskip}
\def\qed{\ifvmode\mbox{ }\else\unskip\fi\hskip 1em plus 10fill$\Box$}
\newenvironment{proofof}[1]
      {\medskip\noindent{\bf Proof of #1:}\hspace{1mm}}
      {\hfill$\Box$\medskip}
\input{epsf}
\makeatletter
\def\Ddots{\mathinner{\mkern1mu\raise\p@
\vbox{\kern7\p@\hbox{.}}\mkern2mu
\raise4\p@\hbox{.}\mkern2mu\raise7\p@\hbox{.}\mkern1mu}}
\makeatother

\title{\vspace{-0.7cm}Short proofs of some extremal results III}
\author{David Conlon\thanks{Department of Mathematics, California Institute of Technology, Pasadena, CA 91125. Email: {\tt dconlon@caltech.edu}. Research
supported in part by ERC Starting Grant 676632.}\and
Jacob Fox\thanks{Department of Mathematics, Stanford University, Stanford, CA 94305. Email: {\tt jacobfox@stanford.edu}. Research supported by a Packard Fellowship and by NSF Award DMS-1855635.}
\and
Benny Sudakov\thanks{Department of Mathematics, ETH, 8092 Zurich, Switzerland.
Email: {\tt benjamin.sudakov@math.ethz.ch}. Research supported by SNSF grant 200021-175573.}}
\date{}

\begin{document}
\maketitle

\begin{abstract}
We prove a selection of results from different areas of extremal combinatorics, including complete or partial solutions to a number of open problems. These results, coming mainly from extremal graph theory and Ramsey theory, have been collected together because in each case the relevant proofs are reasonably short. 
\end{abstract}

\section{Introduction}

We study several questions from extremal combinatorics, a broad area of discrete mathematics which deals with the problem of maximizing or minimizing the cardinality of a collection of finite objects satisfying certain properties. The problems we consider come mainly from the areas of extremal graph theory and Ramsey theory, though they also touch upon additive combinatorics and random graphs. In many cases, they give complete or partial solutions to open problems posed by other researchers.

While each of the results in this paper is interesting in its own right, the proofs are all quite short. Accordingly, in the spirit of Alon's `Problems and results in extremal combinatorics' papers~\cite{Al1, Al2, Al3} and our own earlier papers~\cite{CFS14, CFS16}, we have combined them. We describe the results in brief below, though we refer the reader to the relevant section for full details on each topic. Each section is intended to be self-contained and may be read separately from all others.

In Section~\ref{sec:ind}, we study the extremal problem of how sparse a graph must be to guarantee that there is an independent set of order $k$ whose vertices form an arithmetic progression and present applications of our result to several questions in Ramsey theory. In Section~\ref{sec:joint}, we answer a question of Mubayi~\cite{Mubayi} about the interplay between the local and global counts of copies of the clique $K_r$ in graphs with at least the extremal number of edges. In Section~\ref{sec:inf}, we study the distribution of edges in infinite $K_{s,t}$-free graphs, solving a problem posed by Cilleruelo~\cite{C16}. Section~\ref{sec:ramcon} contains a short study of the concentration of Ramsey numbers of random graphs, while, in Section~\ref{sec:mult}, we show that the Ramsey multiplicity of a graph may change substantially when we increase the number of colors from two to three. Finally, we partially answer a Ramsey problem of F\"uredi, Gy\'arf\'as and Simonyi~\cite{FuGySi} on connected matchings in Section~\ref{sec:match}.

All logarithms are base $2$ unless otherwise stated. For the sake of clarity of presentation, we systematically omit floor and ceiling signs whenever they are not crucial. We also do not make any serious attempt to optimize absolute constants in our statements and proofs. 

\section{Independent arithmetic progressions} \label{sec:ind}

A classical theorem of Tur\'an \cite{Tu} shows that any graph on $n$ vertices with less than $\frac{n(n-k+1)}{2(k-1)}$ edges contains an independent set of order $k$. The celebrated Szemer\'edi's theorem \cite{Sz} states that for $\delta>0$, $k \in \mathbb{N}$, and $n$ sufficiently large in terms of $k$ and $\delta$, any subset of $[n]=\{1,\ldots,n\}$ of order at least $\delta n$ contains a $k$-term arithmetic progression. Here we marry the themes of these results and deduce as consequences bounds on three other well-studied problems on rainbow arithmetic progressions and set mappings. 

Given a graph with vertex set $[n]$, a $k$-term arithmetic progression is said to be {\it independent} if it is an independent set in the graph. Our main result is a Tur\'an-type theorem, showing that any sparse graph on vertex set $[n]$ contains an independent arithmetic progression. Before proving this result, we need a standard estimate from number theory. Note that all logs will be taken to base $e$.

\begin{lemma} \label{lem:sieve}
There is a positive constant $\eta$ such that, for all $n \geq \eta^{-1} k \log k$, the number of integers from $[n]$ which are relatively prime to $1, 2, \dots, k$ is at least $\eta n/\log k$.
\end{lemma}

\begin{proof}
Writing $\Phi(x, y)$ for the number of integers less than or equal to $x$ all of whose prime factors are greater than $y$, a result of Buchstab (see Section 7.2 of~\cite{MV06}) says that
\[\Phi(x,y) = \frac{w(u) x}{\log y} - \frac{y}{\log y} + O\left(\frac{x}{\log^2 x}\right),\]
where $u$ is defined by $y = x^{1/u}$ and $w(u)$ is the Buchstab function, equal to $1/u$ for $1 < u \leq 2$ and asymptotic to $e^{-\gamma}$, with $\gamma$ the Euler--Mascheroni constant, as $u$ tends to infinity. For $k$ sufficiently large, say $k \geq k_0$, and $n \geq k \log k$, the required estimate with $\eta = 1/10$ easily follows by applying this result with $x = n$ and $y = k$. For $k < k_0$, the estimate follows by choosing $\eta$ such that $\eta^{-1} \geq \max(20 \log k_0, k_0 \log k_0)$. Then $n \geq k_0 \log k_0$, so that $\Phi(n, k) \geq \Phi(n, k_0) \geq n/10 \log k_0 \geq \eta n/\log k$.
\end{proof}

Our main result, which is tight up to the logarithmic factor, is now as follows.

\begin{theorem}\label{thm:APmain}
There is a positive constant $\varepsilon$ such that any graph $G$ on $[n]$ with fewer than $\varepsilon \frac{n^2}{k^2 \log k}$ edges contains a $k$-term independent arithmetic progression.
\end{theorem}

\begin{proof} 
We split into two cases, depending on the size of $n$. For $n \geq 2 \eta^{-1} k^2 \log k$, where $\eta$ is as in Lemma~\ref{lem:sieve}, we consider the set of integers $X$ which are relatively prime to $1,2,\ldots,k$ and let $\mathcal{A}$ be the set of $k$-term arithmetic progressions in $[n]$ whose difference is in $X$. We can form an arithmetic progression in $\mathcal{A}$ by choosing the first term from $[n/2]$ and the common difference from $X \cap [n/2k]$. Therefore, since $n/2k \geq \eta^{-1} k \log k$, Lemma~\ref{lem:sieve} applies to show that $|\mathcal{A}| \geq \eta n^2/4k \log k$. Each pair of integers are in arithmetic progressions with at most one common difference in $X$ and, hence, are in at most $k-1$ arithmetic progressions in $\mathcal{A}$. Thus, the number of arithmetic progressions in $\mathcal{A}$ which contain an edge of $G$ is at most $e(G)k$. Taking $\varepsilon < \eta/8$, we have that $e(G)k <  \varepsilon \frac{n^2}{k \log k} < |\mathcal{A}|$, so there is an arithmetic progression in $\mathcal{A}$ which forms an independent set. 

For the second case, when $n < 2 \eta^{-1} k^2 \log k$, we let $\mathcal{B}$ be the set of $k$-term arithmetic progressions in $[n]$ whose difference is a prime. By the same argument as in the previous case, the number of arithmetic progressions in $\mathcal{B}$ which contain an edge of $G$ is at most $e(G)k < \varepsilon \frac{n^2}{k \log k}$. On the other hand, the number of progressions in $\mathcal{B}$ is at least $\pi(n/2k) n/2$, where $\pi(x)$ is the prime counting function. Since there exist positive constants $a$ and $C$ such that $\pi(x) > a \frac{x}{\log x}$ and $2 \eta^{-1} k^2 \log k < k^C$, we have that $\pi(n/2k) n/2 > \frac{a}{4C} \frac{n^2}{k \log k}$. Therefore, for $\varepsilon < a/(4C)$, there is an independent arithmetic progression.
\end{proof}

In a coloring of $[n]$, an arithmetic progression is {\it rainbow} if its elements are all different colors. The {\it sub-Ramsey number $sr(m,k)$} is the minimum $n$ such that every coloring of $[n]$ in which no color is used more than $m$ times has a 
rainbow $k$-term arithmetic progression. Alon, Caro, and Tuza \cite{ACT} proved that there are constants $c,c'>0$ such that $$c'\frac{mk^2}{\log mk} \leq sr(m,k) \leq c mk^2 \log (mk).$$ They also showed that there is an upper bound on $sr(m,k)$ 
which is linear in $m$ but with a worse dependence on $k$, namely, $sr(m,k) \leq cmk^3$. The lower bound was later improved by Fox, Jungi\'c, and Radoi\v ci\' c \cite{FJR} to $sr(m,k) \geq c'mk^2$. Here we improve on the upper bound of Alon, Caro, and Tuza \cite{ACT}. 

\begin{corollary}\label{subRamsey}
There is a constant $c$ such that the sub-Ramsey number satisfies 
$$sr(m,k) \leq cmk^2 \log k.$$
\end{corollary}

\begin{proof}
Consider a coloring of $[n]$ with $n=  \varepsilon^{-1} mk^2 \log k$, with $\varepsilon$ as in Theorem~\ref{thm:APmain}, where no color appears more than $m$ times. Define a graph on $[n]$ where two integers are adjacent if they receive the same color. The graph consists of a disjoint union of cliques of order at most $m$. Since the maximum of $\sum_i \binom{x_i}{2}$ under 
the constraint $\sum_i x_i=n$ occurs when each term is as large as possible, the number of edges in this graph is at most $\frac{n}{m}{m \choose 2}<\frac{nm}{2}$. Therefore, by our choice of $n$, the number of edges is such that Theorem \ref{thm:APmain} applies to give an independent $k$-term arithmetic progression, which is a rainbow arithmetic progression in our coloring of $[n]$. 
\end{proof}

Let $T_k$ denote the smallest positive integer $t$ such that for every positive integer $m$, every equinumerous $t$-coloring of $[tm]$ contains a rainbow $k$-term arithmetic progression. Jungi\'c et al.~\cite{JLMNR} proved that there are positive constants $c,c'$ such that $$c' k^2 \leq T_k \leq c k^3.$$ 
They conjectured that the lower bound is correct, that is, $T_k=\Theta(k^2)$, a problem which was reiterated in the survey~\cite{JNR}. 
Here we make progress on this conjecture, improving the upper bound to $c k^2 \log k$. Note that an equinumerous 
$t$-coloring of $[tm]$ uses each color exactly $m$  times, so $T_k$ is at most the maximum of $sr(m,k)/m$ over all positive integers $m$. Hence, by Corollary \ref{subRamsey}, we obtain the following corollary. 

\begin{corollary}
\label{T_k}
There is a constant $c$ such that
$$T_k \leq c k^2 \log k.$$
\end{corollary}

Motivated by the set mapping problem of Erd\H{o}s and Hajnal, Caro \cite{C87} proved that for every positive integer $k$, there is a minimum integer $n_0 = n_0(k)$ such that, for all $n \geq n_0$ and every permutation $\pi:[n] \to [n]$, there is a $k$-term arithmetic progression $A$ such that $\pi(i) \not \in A$ for all $i \in A$.  Moreover, he showed that there are constants $c, c' > 0$ such that $c' k^2/\log k \leq n_0(k) \leq k^2 2^{c \log k / \log \log k}$. Alon et al.~\cite{ACT} used the same methods they had used to bound $sr(m,k)$ to improve the earlier upper bound to $n_0(k) \leq c k^2 \log k$. Our result gives a simple alternative proof of this.

\begin{corollary}
There is a constant $c$ such that
$$n_0(k) \leq c k^2 \log k.$$
\end{corollary}

\begin{proof}
Consider the graph on $[n]$ with edges $(i,\pi(i))$ for $i \in [n]$. This graph has at most $n$ edges. By choosing $c$ large enough, we can make the number of edges such that Theorem \ref{thm:APmain} applies to give an independent arithmetic progression in this graph. This arithmetic progression has the required property.
\end{proof}

\section{Joints and cliques} \label{sec:joint}

Let $t_r(n)$ denote the number of edges in the balanced complete $r$-partite graph $T_{n,r}$ with $n$ vertices (so parts differ in size by at most one). Tur\'an's theorem \cite{Tu} says that the maximum possible number of edges in a $K_{r+1}$-free graph on $n$ vertices is $t_r(n)$, with equality if and only if the graph is $T_{n,r}$. Hence, a graph on $n$ vertices with at least $t_r(n)$ edges which is not $T_{n,r}$ must have at least one $K_{r+1}$. Must it have many $(r+1)$-cliques? Must there be an edge in many $(r+1)$-cliques? Questions of this sort have been studied since the early days of extremal graph theory. 

Rademacher, in unpublished work from 1950, proved that every graph on $n$ vertices with $t_2(n)+1$ edges contains at least $\lfloor n/2 \rfloor$ triangles, which is tight by adding an edge in a largest part of a balanced complete bipartite graph.  Erd\H{o}s \cite{Er62} extended this to graphs with a linear number of added edges and in \cite{Er62a} proved an analogous result for larger cliques. The general Erd\H{o}s--Rademacher problem, to determine how many copies of $K_{r+1}$ must be in every graph on $n$ vertices with $m$ edges, has a long and rich history (see, for instance, the recent paper \cite{KLPS} and its references).

An \emph{$r$-joint} is a collection of copies of $K_r$ in a graph that share a common edge. The \emph{joint number $j_r(G)$} is the size of the largest $r$-joint in a graph $G$. Erd\H{o}s \cite{Er62} proved that every graph $G$ on $n$ vertices with more than $t_2(n)$ edges satisfies $j_3(G) \geq n/6 - O(1)$ and conjectured that the $O(1)$ term can be removed. This conjecture, which is tight, was later proved by Edwards and, independently, Khad\v ziivanov and Nikiforov \cite{KN}. For $r>2$, the minimum value of $j_{r+1}(G)$ over all graphs $G$ on $n$ vertices with more than $t_r(n)$ edges is not well understood. Progress on this problem was made by Erd\H{o}s \cite{Er69} in 1969 and, more recently, by Bollob\'as and Nikiforov (see \cite{BN1,BN2,BN3}). 

Mubayi introduced and studied a natural problem which looks at the interplay between the two questions above. Namely, how many triangles must be in every $n$-vertex graph $G$ with more than $t_2(n)$ edges if no edge is in more than $j$ triangles? Denote this number by $t(n,j)$. Mubayi \cite{Mubayi} showed that $t(n,j)$ jumps from quadratic to cubic in $n$ when $j$ goes below the threshold $n/4$. In \cite{CFS}, we studied this problem in more detail and proposed a precise conjecture that states that if $n/6 \leq j \leq n/4$, then every graph on $n$ vertices with at least $\lfloor n^2/4 \rfloor$ edges and no edge in more than $j$ triangles which is not the Tur\'an graph $T_{n,2}$ has at least $j^2(n-4j)$ triangles. Moreover, we conjectured that equality holds if and only if the graph is $S_{j,n}$, a blow-up of the $3$-prism graph, the graph on six vertices which consists of two disjoint triangles with a perfect matching between them, with the parts corresponding to one of the matching edges of size $\lfloor (n-4j)/2\rfloor$ and $\lceil (n-4j)/2 \rceil$ and the other four parts of size $j$. We proved this conjecture at the two extremes, if $j=n/6$ or if $.2495n \leq j \leq n/4$. 

Here we consider the generalization of Mubayi's problem to larger cliques. This question was raised by Mubayi \cite{Mubayi} and again recently by the authors \cite{CFS}. Let $t(n,j,r)$ be the minimum number of copies of $K_{r}$ in every graph $G$ with $n$ vertices, more than $t_{r-1}(n)$ edges and $j_r(G) \leq j$. Extending Mubayi's result, we prove that $t(n,j,r+1)$ has a phase transition from $\Theta(n^{r})$ to $\Theta(n^{r+1})$ when $j \approx \left(\frac{r-1}{r^2}\right)^{r-1}n^{r-1}$. 

\begin{theorem}\label{mainthmjoint}
For each integer $r \geq 2$ and $\eta>0$, there is $\delta>0$ such that the following holds. Let $G$ be a graph on $n$ vertices with at least $t_r(n)$ edges. Then either $G$ is the Tur\'an graph $T_{n,r}$ or $G$ has at least $\delta n^{r+1}$ copies of $K_{r+1}$ or $G$ has an edge in at least $(1-\eta)\left(\frac{r-1}{r^2}\right)^{r-1}n^{r-1}$ copies of $K_{r+1}$. 
\end{theorem}

A weaker result on this problem was proved earlier by Allen et al. (\cite{ABHP}, Lemma 8), who used it to study Tur\'an-type problems in random hypergraphs.

To see that the bound on the joint number in Theorem \ref{mainthmjoint} is asymptotically tight, consider a graph $G=G_{n,r}$ on $n$ vertices with vertex partition $V(G)=V_0 \cup V_1 \cup \ldots \cup V_r$, where $V_0=\{v\}$ consists of a single vertex, all other parts are of size $\lfloor \frac{n-1}{r}\rfloor$ or $\lceil \frac{n-1}{r}\rceil$, there is a complete bipartite graph between $V_i$ and $V_j$ for all $1 \leq i < j \leq r$ and $v$ is adjacent to $\lfloor \frac{r-1}{r^2}(n-1) \rfloor$ or $\lceil \frac{r-1}{r^2}(n-1) \rceil$ vertices in each part so that the total number of edges is $t_r(n)$.  For simplicity of the analysis, assume $n \equiv 1 \pmod{r^2}$, so the number of copies of $K_{r+1}$ in $G$ is $\left(\frac{r-1}{r^2}\right)^r(n-1)^r$ and the joint number is $j_{r+1}(G)=\left(\frac{r-1}{r^2}\right)^{r-1}(n-1)^{r-1}$. Observe that one can add $o(n)$ additional edges from $v$ to get more than $t_r(n)$ edges, while the number of copies of $K_{r+1}$ and the joint number $j_{r+1}(G)$ will be asymptotically unchanged. 

More generally, we can define graphs $G=G_{n,r}(s)$ for which $G_{n,r}$ is the case $s=1$ and such that as we  increase $s$ from $1$, the joint number $j_{r+1}(G)$ decreases, but the number of copies of $K_{r+1}$ increases. It also generalizes the construction that is conjectured by the authors in \cite{CFS} to be tight for $r=2$ and may give an asymptotically tight bound for general $r$. The graph $G$ has $n$ vertices and a vertex partition into $r+1$ parts $V_0,V_1,\dots,V_r$, with $V_0$ of size $s$ and all the other parts of size $\lfloor \frac{n-s}{r} \rfloor$ or $\lceil \frac{n-s}{r} \rceil$. Each $V_i$ for $0 \leq i \leq r$ has an equitable partition into $r$ parts, $V_i=V_{i,1} \cup \dots \cup V_{i,r}$. The induced subgraph of $G$ on $V(G) \setminus V_0$ is complete $r$-partite with parts $V_1,\dots,V_r$. The induced subgraph of $G$ on $V_0$ is a balanced complete $r$-partite graph with parts $V_{0,1},V_{0,2},\dots,V_{0,r}$. For $1 \leq i \leq r$, we have $V_{0,j}$ complete to $V_{i,j'}$ if $j \not = j'$ and otherwise $V_{0,j}$ empty to $V_{i,j}$. 

To prove Theorem \ref{mainthmjoint} we need several lemmas. The first one is the celebrated graph removal lemma, which shows that if a graph has relatively few copies of a graph $H$, then it can be made $H$-free by removing a relatively small number of edges. This lemma was already used by Mubayi, who also remarked that it can be applied to study the case of larger cliques.

\begin{lemma}[Graph Removal Lemma] \label{grl}
For each $\epsilon>0$ and graph $H$ on $k$ vertices, there is $\delta>0$ such that every graph with $n$ vertices and at most $\delta n^k$ copies of $H$ can be made $H$-free by removing at most $\epsilon n^2$ edges. 
\end{lemma}

We will also use the following stability result of F\"uredi. 

\begin{lemma}[F\"uredi~\cite{Fur}] \label{Furlem} 
Any $K_{r+1}$-free graph on $n$ vertices with $t_r(n)-m$ edges can be made $r$-partite by removing at most $m$ edges. 
\end{lemma}

We have the following corollary of Lemmas \ref{grl} and \ref{Furlem} (replacing $\epsilon$ by $\epsilon/2$ in the application of the graph removal lemma, Lemma \ref{grl}). It gives a stability version of Tur\'an's theorem, showing that any graph for which the number of edges is at least the extremal number for a clique $K_{r+1}$ either contains many $K_{r+1}$ or is close to being $r$-partite. 

\begin{corollary}\label{corclose} For each $\epsilon>0$ and positive integer $r$, there is $\delta>0$ such that if a graph on $n$ vertices with at least $t_r(n)$ edges has at most $\delta n^{r+1}$ copies of $K_{r+1}$, then it can be made $r$-partite by removing at most $
\epsilon n^2$ edges. 
\end{corollary}

The proof of Corollary \ref{corclose} described above gives a poor bound for $\delta$ in terms of $\epsilon$ due to the application of the  graph removal lemma. Alternative proofs with a much better, polynomial dependence are known, such as by using the Moon--Moser inequalities (see, e.g., Section 4.3 of Shapira's lecture notes \cite{Shapira}) or through a sampling argument. 

We will also make repeated use of the following lemma, which says that if you decrease the sum of some positive numbers by a small amount and do not increase any of the numbers, then, to minimize the product of the numbers, only the smallest of the numbers should be decreased. 

\begin{lemma}\label{productdecrease}
Suppose $0\leq x_1 \leq x_2 \leq \ldots \leq x_s$ and $0 \leq y_i \leq x_i$ for $1 \leq i \leq s$. 
If $\alpha:=\sum_{i=1}^s (x_i-y_i)$, then $\prod_{i=1}^s y_i \geq (x_1-\alpha)\prod_{i=2}^s x_i$. 
\end{lemma}
\begin{proof}
As the derivative of $\prod_{i=1}^s x_i$ with respect to $x_i$ is the product of the other variables, the derivative is maximized when $i=1$. To minimize the product, we should thus delete as much from $x_1$ as possible, establishing the lemma. 
\end{proof}

The next lemma shows that, for a graph $G$ on $n$ vertices with at least $t_r(n)$ edges, if in addition to assuming $G$ has few copies of $K_{r+1}$, we also assume no edge is in too many $K_{r+1}$, then there is an almost spanning $r$-partite  induced subgraph which is nearly balanced and nearly complete. 

\begin{lemma}\label{firstkey} For each $\epsilon>0$ and positive integer $r$, there is $\delta>0$ such that if a graph $G$ on $n$ vertices with at least $t_r(n)$ edges has at most $\delta n^{r+1}$ copies of $K_{r+1}$ and each edge is in at most  $\left(1-c\right)\left(\frac{n}{r}\right)^{r-1}$ copies of $K_{r+1}$ where $c=2(r^2+r)\epsilon^{1/2}$, then there is a set $S$ of at most $\epsilon^{1/2}n$ vertices whose deletion makes the remaining induced subgraph $r$-partite with each part of size at most $\frac{n}{r}+2\epsilon^{1/2}n$ and at most $\epsilon n^2$ missing edges between the parts.
\end{lemma}

\begin{proof}
Let $\delta$ be as in Corollary \ref{corclose} or $\epsilon^{r+1}$, whichever is smaller. If $\epsilon < 1/n$, then from the assumption of the lemma, the graph $G$ has at most $\delta n^{r+1}<1$ copies of $K_{r+1}$ and hence is $K_{r+1}$-free and, therefore, must be the Tur\'an graph. We could then take $S$ to be the empty set. So we may assume $\epsilon \geq 1/n$.  

By Corollary \ref{corclose}, $G$ has a subgraph $G'$ formed by removing  at most $\epsilon n^2$ edges which is $r$-partite. 
Let $U_1,\ldots,U_{r}$ be the $r$ parts. Since the number of possible edges across the $r$ parts is at  most $t_r(n)$, there are at most $\epsilon n^2$ nonadjacent pairs of vertices across the parts. 

Let $c'=2\sqrt{\epsilon}$. Call a vertex $v \in U_i$ of $G$ {\it normal} if it is adjacent to all but at most $c'n$ vertices in $V(G) \setminus U_i$ and {\it abnormal} otherwise. So the total number of abnormal vertices is at most $2\cdot \epsilon n^2/(c'n)=(2\epsilon/c')n=\epsilon^{1/2}n$. Let $S$ be the set of abnormal vertices, so that $U_i':=U_i \setminus S$ is the set of normal vertices in $U_i$. It suffices now to prove that each $U_i'$ is an independent set (the fact that it has the correct size will fall out in the next paragraph as a byproduct of our argument). 

Suppose for the sake of contradiction that there is an adjacent pair $u,v \in U_i'$ for some $i$. Then $u,v$ are adjacent to all but at most $2c'n$ vertices not in $U_i$. We have $\left||U_j|-\frac{n}{r}\right| \leq 2\sqrt{\epsilon}n$ for each $j$, as otherwise the total number of edges of $G$ is at most 
\begin{eqnarray*} \epsilon n^2+{n \choose 2}-\sum_{j=1}^{r}{|U_j| \choose 2} & = & \left(\frac{1}{2}-\frac{1}{2r}+\epsilon \right)n^2-\frac{1}{2}\sum_{j=1}^{r}\left(|U_j|-\frac{n}{r}\right)^2<\left(\frac{1}{2}-\frac{1}{2r}-\epsilon \right)n^2 \\ & \leq & \frac{r-1}{2r} n^2-n < t_r(n),
\end{eqnarray*} 
contradicting that $G$ has at least $t_r(n)$ edges. The last inequality giving a lower bound on the number of edges of the  balanced complete $r$-partite graph on $n$ vertices follows from a straightforward computation. 

We next provide a lower bound on the number of $K_{r+1}$ that $(u,v)$ is in by giving a lower bound on the number of such cliques with one vertex in each $U_j$ with $j \not = i$. If these $r-1$ vertex subsets were complete to each other, we would get $\prod_{j \not = i} |U_j \cap N(u) \cap N(v)|$ such cliques. Using that $\left||U_j|-\frac{n}{r}\right| \leq 2\sqrt{\epsilon}n$  for each $j$ and 
$$\sum_{j \not = i} |U_j \cap N(u) \cap N(v)| \geq \sum_{j \not =i} |U_j| -2c'n = \sum_{j \not =i} |U_j| -4\sqrt{\epsilon}n,$$ 
this product is at least 

\begin{equation}\label{productbound} \left(\frac{n}{r}- 6\sqrt{\epsilon}n\right) \cdot \left(\frac{n}{r}-2\sqrt{\epsilon}n\right)^{r-2},\end{equation}
which follows from Lemma \ref{productdecrease} by taking $s=r-1$, $x_j=|U_j|$ and $y_j= |U_j \cap N(u) \cap N(v)|$ and relabeling so that the labels are $1,\ldots,r-1$ and the $x_j$ are in increasing order.

Relabel again so that $y_1,\ldots,y_{r-1}$ are in increasing order. For each pair of vertices, one in some $U_j \cap N(u) \cap N(v)$ and the other in $U_{j'} \cap N(u) \cap N(v)$, the number of $(r-1)$-tuples they are in with one vertex in each  $U_{\ell} \cap N(u) \cap N(v)$  with $\ell \not = i,j,j'$ is 
$\prod_{\ell \not = i,j,j'} |U_{\ell} \cap N(u) \cap N(v)|=\prod_{\ell \not = i,j,j'} y_{\ell}$. Let $m$ denote the number of nonadjacent pairs of vertices with one vertex in 
some $U_{j} \cap N(u) \cap N(v)$ and the other in some $U_{j'} \cap N(u) \cap N(v)$ with $j,j',i$ distinct. So
 $m \leq \epsilon n^2$. 

Summing over all $m$ nonadjacent pairs with vertices in different parts, the number of $(r-1)$-tuples with one vertex in each $U_{\ell} \cap N(u) \cap N(v)$ with $\ell \not = i$ which is not a clique is at most $m\prod_{j=3}^{r-1}y_{j}$. Thus, the number of cliques $K_{r+1}$ that $(u,v)$ is in with one vertex in each $U_j$ with $j \not = i$ is at least \begin{eqnarray*} \left(y_1y_2-m\right)\prod_{\ell=3}^{r-1} y_{\ell} & = &  \left(1-\frac{m}{y_1y_2}\right) \prod_{\ell=1}^{r-1} y_{\ell} \geq \left(1-\frac{m}{y_1y_2}\right)\left(\frac{n}{r}- 6\sqrt{\epsilon}n\right) \left(\frac{n}{r}-2\sqrt{\epsilon}n\right)^{r-2} \\ & \geq & \left(\left(\frac{n}{r}- 6\sqrt{\epsilon}n\right) \left(\frac{n}{r}- 2\sqrt{\epsilon}n\right)-m \right) \cdot \left(\frac{n}{r}- 2\sqrt{\epsilon}n\right)^{r-3},
\end{eqnarray*} where the first inequality is by (\ref{productbound}) and the second inequaity uses $y_1y_2 \geq  \left(\frac{n}{r}- 6\sqrt{\epsilon}n\right) \left(\frac{n}{r}- 2\sqrt{\epsilon}n\right)$. This lower bound on the number of $K_{r+1}$ containing the pair $u,v$
is more than
$$\left(1-2(r^2+r)\sqrt{\epsilon}\right)\left(\frac{n}{r}\right)^{r-1}=(1-c)\left(\frac{n}{r}\right)^{r-1},$$
contradicting the assumption of the lemma and completing the proof.
\end{proof}

We will also use a few simple inequalities involving real numbers in $[0,1]$. 

\begin{lemma}
Suppose $0 \leq a_1 \leq a_2 \leq \ldots \leq a_r \leq 1$ and $\sum_i a_i \geq r-1$. Then $\prod_{i=2}^r a_i \geq \left(\frac{r-1}{r}\right)^{r-1}$ with equality if and only if all $a_i$ are equal to $(r-1)/r$.
\end{lemma}

\begin{proof}
By scaling each $a_i$ by a factor $(r-1)/\sum_{i=1}^r a_i$, we may assume $\sum_{i=1}^r a_i=r-1$. By averaging, we must have $a_1\leq (r-1)/r<1$. If $a_1=0$, then $a_2=a_3=\ldots=a_r$, in which case $\prod_{i=2}^r a_i = 1$. Hence, we may assume $0<a_1<1$. 

Thus, there is a positive integer $j \leq r-1$ such that $\frac{j-1}{j} < a_1 \leq \frac{j}{j+1}$. Fixing $a_1$, by concavity of the function $x(C-x)$ for any fixed $C$, the product $\prod_{i=2}^r a_i $ is minimized when $a_1=a_2=\dots=a_j$ and $a_i=1$ for $i>j+1$, in which case $a_{j+1}=j(1-a_1)$ and $\prod_{i=2}^r a_i = a_1^{j-1}j(1-a_1)$. The derivative of this function with respect to $a_1$ is negative in this range and hence the function is minimized when $a_1=\frac{j}{j+1}$, in which case $\prod_{i=2}^r a_i = \left(\frac{j}{j+1}\right)^{j}$. This is a decreasing function of $j$, so it is minimized when $j=r-1$ and we get $\prod_{i=2}^r a_i \geq \left(\frac{r-1}{r}\right)^{r-1}$. 
\end{proof}

We have the following corollary. 

\begin{corollary}\label{corbelow}
Suppose $0 \leq a_1 \leq a_2 \leq \ldots \leq a_r \leq 1$ and $\sum_i a_i \geq r-1$. Suppose $0 \leq b_i \leq a_i$ for $1 \leq i \leq r$ and $\alpha:=\sum_{i=1}^r (a_i-b_i)\leq 1/4$.  Then $\prod_{i=2}^r b_i \geq \left(1-2\alpha\right)\left(\frac{r-1}{r}\right)^{r-1}$.  
\end{corollary}
\begin{proof}
As $\sum_i a_i \geq r-1$, $a_1 \leq a_2$ and each $a_i \leq 1$, it follows that $a_2 \geq 1/2$. As $\alpha \leq 1/4$, we have $a_2-\alpha \geq (1-2\alpha)a_2$. We have $\prod_{i=2}^r b_i \geq (a_2-\alpha)\prod_{i \geq 3} a_i \geq (1-2\alpha)\prod_{i \geq 2} a_i \geq (1-2\alpha)\left(\frac{r-1}{r}\right)^{r-1}$, where the first inequality is by Lemma \ref{productdecrease} and the 
last inequality is by the previous lemma. 
\end{proof}

{\bf Proof of Theorem \ref{mainthmjoint}:} We may suppose $r \geq 3$ since the case $r=2$ is handled in \cite{CFS} and \cite{Mubayi}.
Let $\epsilon = \left(\frac{\eta}{20r^2}\right)^2$ (so  $\eta=20r^2\epsilon^{1/2}$) and let $\delta=\delta(r, \epsilon)$ be a constant satisfying the assertion of Lemma \ref{firstkey}. 
If $G$ has at least $\delta n^{r+1}$ copies of $K_{r+1}$ or an edge contained in at least $\left(1-2(r^2+r)\epsilon^{1/2}\right)\left(\frac{n}{r}\right)^{r-1}$ copies of $K_{r+1}$, then we are done.
Otherwise, by Lemma \ref{firstkey}, we may assume that $G$ has a vertex partition $V(G)=S \cup V_1 \cup \ldots \cup V_{r}$ where $|S| \leq \epsilon^{1/2} n$, each $V_i$ is an independent set with $|V_i| \leq \frac{n}{r}+2\epsilon^{1/2}n$ and there are at most $\epsilon n^2$ missing edges between (but not inside) the parts $V_1,\ldots,V_r$. 

If there is $i \in [r]$ and a vertex in $S$ not adjacent to any vertex in $V_i$, then move this vertex from $S$ to $V_i$. Letting $V_1$ be the part of minimum order, if there is a vertex $v$ in $S$ of degree less than $\sum_{i=2}^r |V_i|$, then we move $v$ to $V_1$ and replace its neighborhood by $\bigcup_{i=2}^r V_i$. This increases the number of edges while guaranteeing that $v$ participates in zero copies of $K_{r+1}$, but also maintains the conditions of the lemma and partition (aside from possibly increasing the order of $V_1$). If there is a vertex $u$ in some $V_i$ nonadjacent to more than $|S|$ vertices in $\bigcup_{j \not = i} V_j$, then we can delete its edges to $S$ and make it complete to  $\bigcup_{j \not = i} V_j$. This increases the number of edges while guaranteeing that $u$ is not in any copy of $K_{r+1}$ and again maintains the conditions of the lemma and partition. 

We repeat the above operations until no such vertex remains. We arrive at a graph (which we again call $G$) on $n$ vertices with at least $t_r(n)$ edges and a vertex partition $V(G)=S \cup V_1 \cup \ldots \cup V_{r}$ (which we relabel if necessary) satisfying 
\begin{itemize}
\item $|V_1| \leq |V_2| \leq \dots \leq |V_r|$, each $V_i$ is an independent set and has size at most $\frac{n}{r}+3\epsilon^{1/2}n$ .
\item $|S| \leq \epsilon^{1/2} n$, each vertex in $S$ has a neighbor in every $V_i$ and has degree at least $\sum_{i=2}^r |V_i|$.
\item Each vertex in $V_i$ for $i \in [r]$ is adjacent to all but at most $|S|$ vertices in $\bigcup_{j \not = i} V_j$ and there are at most $\epsilon n^2$ missing edges between (but not inside) the parts $V_1,\ldots,V_r$.
\end{itemize}

If $S$ is empty, then $G$ is $r$-partite and, since it has at least $t_r(n)$ edges, $G$ must be the balanced complete $r$-partite graph. So we can assume $S$ is nonempty.

Since $S$ is nonempty, there is a vertex $v \in S$ which has a neighbor in each $V_i$. Let $i_0$ be such that $|N(v) \cap V_{i_0}|$ is minimum and let $w \in N(v) \cap V_{i_0}$. We claim that the edge $(v,w)$ is in at least  $(1-\beta)\left(\left(\frac{r-1}{r^2}\right)n\right)^{r-1}$ copies of $K_{r+1}$. Indeed, as discussed above, $w$ is adjacent to all but at most $|S|$ vertices in $\bigcup_{i \not = i_0} V_i$. The number of vertices in $\bigcup_{i} V_i$ the vertex $v$ is adjacent to is at least 
$$\sum_{i \geq 2}|V_i|-|S| \geq \left(1-\frac{1}{r}\right)(n-|S|)-|S| \geq \left(1-\frac{1}{r}\right)n-2|S|.$$ 

Let $a_i= |N(v)\cap V_i|/|V_r|$, so $0  < a_i \leq 1$ for each $i$ and $a_{i_0}$ is the minimum $a_i$. Observe that $$|V_r|\sum_{i} a_i \geq \left(1-\frac{1}{r}-2\epsilon^{1/2}\right)n \geq  \left(1-\frac{1}{r}-2\epsilon^{1/2}\right)|V_r|/\left(\frac{1}{r}+3\epsilon^{1/2}\right)\geq \left(r-1-4r^2\epsilon^{1/2}\right)|V_r|.$$

Let $b_i=|N(v) \cap N(w) \cap V_i|/|V_r|$ for $i \neq i_0$ and $b_{i_0} = a_{i_0}$, so  $$\sum_{i} b_i \geq \sum_i a_i-|S|/|V_r| \geq \sum_i a_i -r|S|/(n-|S|) \geq \sum_i a_i -2r\epsilon^{1/2}.$$

By possibly increasing some of the $a_i$, keeping their order and keeping each at most $1$, we can guarantee that their sum is increased by at most $4r^2\epsilon^{1/2}$ and is at least $r-1$. Then the difference between the sum of the new $a_i$ and the sum of the $b_i$ is at most $2r\epsilon^{1/2}+ 4r^2\epsilon^{1/2} \leq 5r^2\epsilon^{1/2}=:\alpha$. Hence, it follows from Corollary \ref{corbelow} that $\prod_{i \not = i_0} b_i \geq \left(1-2\alpha\right)\left(\frac{r-1}{r}\right)^{r-1}$. For $i \not = i_0$, let $A_i=N(v) \cap N(w) \cap V_i$, so the $r-1$ parts $A_i$ with $i \not = i_0$ form a complete $(r-1)$-partite subgraph apart from at most $\epsilon n^2$ missing edges. We have $$\prod_{i \not = i_0} |A_i|=\prod_{i \not = i_0} b_i|V_r|=|V_r|^{r-1}\prod_{i \not = i_0} b_i \geq 
|V_r|^{r-1}\left(1-2\alpha\right)\left(\frac{r-1}{r}\right)^{r-1}.$$

Each copy of $K_{r-1}$ with one vertex in each $A_i$ with $i \not = i_0$ together with $v,w$ forms a copy of $K_{r+1}$ containing the edge $(v,w)$. If the $(r-1)$-partite graph between the $A_i$ with $i \not = i_0$ was complete, then we would have that $v,w$ are in at least $\left(1-2\alpha\right)\left(\frac{r-1}{r}\right)^{r-1}|V_r|^{r-1}$ copies of $K_{r+1}$. 

If $r=3$, each missing edge between different $A_i$ with $i \not  = i_0$ decreases the total count of $K_{r+1}$ containing $(v,w)$ by one. Thus, if $r=3$, we have that $v,w$ extend to at least 
$$\left(1-2\alpha\right)\left(\frac{2}{3}\right)^2|V_r|^2-\epsilon n^2 \geq \left(1-2\alpha\right)\left(\frac{2}{3}\right)^2\left(\frac{n-|S|}{3}\right)^2-\epsilon n^2  \geq \left(1-100\epsilon^{1/2}\right)\frac{4}{81}n^2$$ 
copies of $K_4$. 

If $r>3$, as $|A_i| \leq |V_i| \leq |V_r|$, each missing edge between different $A_i$ with $i \not = i_0$ decreases the total count of $K_{r+1}$ containing the edge $(v,w)$ by at most $|V_r|^{r-3}$. Since there are at most $\epsilon n^2$ such missing edges, the total count of $K_{r+1}$ containing the edge $(v,w)$ is at least 
\begin{eqnarray*}\left(1-2\alpha\right)\left(\frac{r-1}{r}\right)^{r-1}|V_r|^{r-1}-\epsilon n^2 \cdot |V_r|^{r-3} & = & \left(\left(1-2\alpha\right)\left(\frac{r-1}{r}\right)^{r-1}-\frac{\epsilon n^2}{|V_r|^2}\right)|V_r|^{r-1} \\ & \geq & \left(\left(1-2\alpha\right)\left(\frac{r-1}{r}\right)^{r-1}-\frac{\epsilon r^2n^2}{(n-|S|)^2}\right)|V_r|^{r-1} \\ & \geq & 
(1-3\alpha)\left(\frac{r-1}{r}\right)^{r-1}|V_r|^{r-1} \\ & \geq & 
(1-3\alpha)\left(\frac{r-1}{r}\right)^{r-1}r^{1-r}(n-|S|)^{r-1} \\ & \geq & 
(1-20r^2\epsilon^{1/2})(r-1)^{r-1}r^{2-2r}n^{r-1},
\end{eqnarray*}
completing the proof. \qed

\section{Infinite $K_{s,t}$-free graphs} \label{sec:inf}

A \emph{Sidon set} is an infinite set of natural numbers such that there are no non-trivial solutions to the equation $a + b = c + d$. It is well known that there are Sidon sets $S$ for which the intersection of $S$ with $\{1, 2, \dots, n\}$ is of order $(1 + o(1)) \sqrt{n}$. However, a result of Erd\H{o}s (see~\cite[Chapter II, \S 3]{HR83}) shows that, for any fixed Sidon set $S$,
\[\liminf_{n \rightarrow \infty} \frac{|S \cap \{1, 2, \dots, n\}|}{\sqrt{n/\log n}} < \infty,\]
where we will take $\log$ to be the natural logarithm  for the remainder of this section. On the other hand, a remarkable construction due to Ruzsa~\cite{Ru98} shows that there are Sidon sets $S$ such that $|S \cap \{1, 2, \dots, n\}| \geq n^{\sqrt{2} - 1 + o(1)}$ for all $n$.

In graphs, the natural analogue of a Sidon set is a $C_4$-free graph. However, it is easy to construct $C_4$-free graphs on vertex set $\mathbb{N}$ such that the induced subgraph on $\{1, 2, \dots, n\}$ has $\Omega(n^{3/2})$ edges for all $n$. For example, one may take the ordered disjoint union of a sequence of $C_4$-free graphs, with the $k^{th}$ graph having $n_k = 2^k$ vertices and $\Omega(n_k^{3/2})$ edges. But if one changes the question to demand that the minimum degree of the induced subgraph on $\{1, 2, \dots, n\}$ is always large, the problem again becomes a difficult one.

By modifying Ruzsa's construction, Cilleruelo~\cite{C16} was able to show that there are graphs $G$ on $\mathbb{N}$ for which the minimum degree of the induced subgraph $G_n$ on $\{1, 2, \dots, n\}$ is at least $n^{\sqrt{2} - 1 + o(1)}$. However, it was left open to determine whether $\liminf_{n \rightarrow \infty} \delta(G_n)/\sqrt{n} = 0$ for all infinite countable $C_4$-free graphs $G$. We answer this question by proving a graph-theoretic analogue of Erd\H{o}s' result, saying that if $G$ is a $C_4$-free graph on $\mathbb{N}$, then
\[\liminf_{n \rightarrow \infty} \frac{\delta(G_n)}{\sqrt{n/\log n}} < \infty.\]

In fact, we will prove a more general result, which applies to $K_{s,t}$-free graphs for any $2 \leq s \leq t$. Here it is known~\cite{KRS96} that if $t$ is sufficiently large compared to $s$, then there are $K_{s,t}$-free graphs on $n$ vertices with $\Omega(n^{2-1/s})$ edges and this can again be lifted to show that there are graphs on $\mathbb{N}$ such that the induced subgraph on $\{1, 2, \dots, n\}$ has $\Omega(n^{2 - 1/s})$ edges for all $n$. Our result says that one cannot guarantee that the minimum degree of these induced subgraphs is $\Omega(n^{1-1/s})$ for all $n$.

\begin{theorem}
Suppose that $2 \leq s \leq t$ and $G$ is a $K_{s,t}$-free graph on $\mathbb{N}$. Then
\[\liminf_{n \rightarrow \infty} \frac{\delta(G_n)}{n^{1-1/s}/(\log n)^{1/s}} < \infty,\]
where $G_n$ denotes the induced subgraph of $G$ on $\{1, 2, \dots, n\}$.
\end{theorem}

\begin{proof}
Fix a large integer $n$ and, for each natural number $\ell$, let $I_\ell = ((\ell - 1)n, \ell n]$ and $J_\ell = \bigcup_{j = 1}^\ell I_j = [1, \ell n]$. Our aim will be to show that for some $\ell \leq n$, there exists a vertex $v \in I_1$ whose degree in $J_\ell$ is less than $16 t^{1/s} s |J_\ell|^{1-1/s}/(\log n)^{1/s}$. For the sake of contradiction, we will assume that this is not the case.

For each $\ell$, let $E_\ell$ denote the number of edges between $I_1$ and $I_\ell$ (for $\ell = 1$, this counts all edges in the graph induced on $I_1$ twice) and let $F_\ell = \sum_{j = 1}^\ell E_j$. For each vertex $v$, let $d_v$ be the number of neighbors of $v$ in $I_1$. Note that 
\[\sum_{v \in J_n} \binom{d_v}{s} \leq (t-1) \binom{n}{s}.\]
Otherwise, there would be $s$ vertices in $I_1$ whose joint neighborhood contained at least $t$ elements, thus forming a $K_{s,t}$. By convexity, we have
\[\sum_{v \in I_j} \binom{d_v}{s} \geq n \binom{E_j/n}{s} = \frac{E_j (E_j - n) \cdots (E_j - (s-1)n)}{s! n^{s-1}}\]
and, therefore,
\[\sum_{j = 1}^n \frac{[E_j (E_j - n)\cdots(E_j - (s-1)n)]_+}{s!n^{s-1}} \leq \sum_{v \in J_n} \binom{d_v}{s} \leq (t-1) \binom{n}{s},\]
where $x_+ = \max\{0,x\}$.
If $E_j \geq sn$, we have $E_j(E_j - n) \cdots (E_j - (s-1)n) \geq s! E_j^s/s^s$. 
Thus,
\[\sum_{j=1}^n E_j^s \leq \sum_{\{j : E_j < sn\}} (sn)^s + \sum_{\{j : E_j \geq sn\}} E_j^s \leq s^s n^{s+1} + \frac{s^s}{s!} \sum_{j = 1}^n [E_j (E_j - n) \cdots (E_j - (s-1)n)]_+ \leq t s^s n^{2s-1}.\]
By H\"older's inequality,
\[\sum_{j = 1}^n \frac{E_j}{j^{(s-1)/s}} \leq \left( \sum_{j=1}^n E_j^s \right)^{1/s} \left( \sum_{j=1}^n \frac{1}{j} \right)^{(s-1)/s} < 2 t^{1/s} s n^{2 - 1/s} (\log n)^{(s-1)/s}.\]
On the other hand,
\[\sum_{j = 1}^n \frac{E_j}{j^{(s-1)/s}} = \sum_{j = 1}^n \frac{F_j - F_{j-1}}{j^{(s-1)/s}}
 = \sum_{j=1}^{n-1} F_j \left(\frac{1}{j^{(s-1)/s}} - \frac{1}{(j+1)^{(s-1)/s}}\right) + \frac{F_n}{n^{(s-1)/s}}
 \geq \sum_{j=1}^{n-1} \frac{F_j}{2 (j+1)^{2 - 1/s}},\]
 where we used that
 \[(j+1)^{(s-1)/s} - j^{(s-1)/s} = j^{(s-1)/s}\left(\left(1 + \frac{1}{j}\right)^{(s-1)/s} - 1\right) \geq j^{(s-1)/s} \left(1 + \frac{s-1}{sj}- 1\right) \geq \frac{1}{2 j^{1/s}}.\]
Since we are assuming that every vertex in $I_1$ has degree at least $16 t^{1/s} s |J_j|^{1-1/s}/(\log n)^{1/s}$ in $J_j$ and $|I_1| = n$ and $|J_j| = jn$, this implies that 
\begin{align*}
\sum_{j = 1}^n \frac{E_j}{j^{(s-1)/s}}  \geq \sum_{j=1}^{n-1} \frac{F_j}{2 (j+1)^{2 - 1/s}}  & \geq \sum_{j=1}^{n-1} \frac{n \cdot 16 t^{1/s} s (jn)^{1-1/s}/(\log n)^{1/s}}{2 (j+1)^{2-1/s}} \geq \frac{16 t^{1/s} s n^{2 - 1/s}}{(\log n)^{1/s}} \sum_{j=1}^{n-1} \frac{j^{1-1/s}}{2 (j+1)^{2-1/s}}\\ &
\geq \frac{4 t^{1/s} s n^{2-1/s}}{(\log n)^{1/s}} \sum_{j=1}^{n-1} \frac{1}{j+1} \geq 2 t^{1/s} s n^{2-1/s} (\log n)^{(s-1)/s}.
\end{align*}
But this is a contradiction, completing the proof.
\end{proof}

A natural generalization of this problem is to consider what happens for infinite $H$-free graphs when $H$ is a fixed bipartite graph. Given a bipartite graph $H$ whose Tur\'an number ex$(n, H)$ is $\Theta(n^{1 + \alpha})$, it is reasonable to ask whether, for any infinite $H$-free graph $G$, 
\[\liminf_{n \rightarrow \infty} \frac{\delta(G_n)}{n^{\alpha}} = 0.\]
Our result proves that this is the case when $H = K_{s,t}$. It would be interesting to find other examples of bipartite graphs $H$ where it holds. For instance, a particular case where we believe that the limit should be zero is when $H = C_{2k}$ for $k \geq 3$. In the other direction, it would be interesting to find examples of bipartite graphs where the limit above is constant. 

\section{Concentration for Ramsey numbers of random graphs} \label{sec:ramcon}

The \emph{Ramsey number} $r(H)$ of a graph $H$ is the smallest natural number $n$ such that any two-coloring of the edges of $K_n$ contains a monochromatic copy of $H$. Our interest here will be in the behavior of this function for random graphs. That is, what can we say about the Ramsey number of the \emph{binomial random graph} $G_{n,p}$, where each edge in an $n$-vertex graph is chosen independently with probability $p$? While the Ramsey properties of random graphs have been studied in great detail (see~\cite{C14} and its references), the problem of estimating Ramsey numbers of random graphs themselves has received surprisingly scant attention (though see~\cite{C13, FS092} for some results). Our main result here is to show that these Ramsey numbers are well concentrated, reducing the problem, at least in first approximation, to the study of $\mathbb{E}(\log r(G_{n,p}))$.

\begin{theorem} \label{thm:ramcon}
For any $p := p(n)$,
\[\mathbb{P}[|\log r(G_{n,p}) - \mathbb{E}(\log r(G_{n,p}))| = \omega(\sqrt{n} \log n)] = o(1).\]
\end{theorem} 
That is, there exists a function $f(n, p)$ such that $r(G_{n,p}) = 2^{f(n,p) \pm \omega(\sqrt{n} \log n)}$ with high probability. A reasonable conjecture is that $f(n,p) = c(p) n$ for some function $c(p)$, perhaps even linear in $p$, but proving this seems well beyond current methods. 

The proof of Theorem~\ref{thm:ramcon} relies on the following simple lemma. Note that we write $v(H)$ for the number of vertices in $H$.

\begin{lemma} \label{lem:bounddiff}
For any graph $H$ and any graph $H'$ obtained from $H$ by removing a single vertex,
\[r(H') \leq r(H) \leq 2 v(H') r(H').\]
\end{lemma}

\begin{proof}
Suppose that we have a complete graph with $2 v(H') r(H')$ vertices whose edges have been colored red and blue. Fix a vertex $v$. By the pigeonhole principle, $v$ has at least $v(H') r(H')$ neighbors in one color, say red. Call this set of neighbors $X$. If there is a vertex $w$  with at least $r(H')$ blue neighbors in $X$, then this neighborhood contains a monochromatic copy of $H'$, which can be extended to a monochromatic copy of $H$ by adding either $v$ or $w$. Otherwise, the induced graph on vertex set $X$ has blue degree at most $r(H')-1$, so any $v(H') - 1$ vertices in $X$ have a red common neighbor in $X$. Hence, $X$ contains a red copy of $H'$, which together with $v$ gives a red copy of $H$.
\end{proof}

We also need the following variant of the Azuma--Hoeffding inequality~\cite[Theorem 1.2]{Warnke}.

\begin{lemma}
\label{azuma}
Let $X_1, \ldots, X_n$ be independent random variables with $X_i$ taking values in the set $\Lambda_i$. Let $\Gamma \subseteq \prod_i \Lambda_i$ be an
event and assume that $f : \prod_i \Lambda_i \rightarrow \mathbb{R}$ satisfies the following Lipschitz condition: there are numbers $c_i \leq d_i$ for $1 \leq i \leq n$ such that, whenever 
$x, x' \in \prod_i \Lambda_i$ differ only in the $k$th coordinate, then $|f(x)-f(x')| \leq c_k$ if $x \in \Gamma$ and $|f(x)-f(x')| \leq d_k$ otherwise.
Then, for any numbers $0 <  \gamma_k \leq 1$ and $t \geq 0$, 
$$\mathbb{P}[|f(x) - \mathbb{E}f(x)| > t] \leq 2\exp\left(-\frac{t^2}{2\sum_i (c_i+\gamma_i(d_i-c_i))^2}\right)+ 2 \, \mathbb{P}[x \not \in \Gamma] \sum_i \gamma_i^{-1}.$$
\end{lemma}

\begin{proofof}{Theorem~\ref{thm:ramcon}}
We encode the edges of the random graph $G_{n, p}$ in a sequence of independent random variables $X_1, \ldots, X_n$, where $X_i \in \Lambda_i=\{0,1\}^{i-1}$ records the edges between the  $i$th vertex and the preceding vertices $1, 2, \dots, i-1$. Let 
$f(G)=\log r(G_{n,p})$. Then, by Lemma~\ref{lem:bounddiff},  for any two graphs $G$, $G'$ which differ only in the edges incident with the $i$th vertex we have  $|f(G) - f(G')| \leq 2 \log (2n) \leq 4 \log n=c_i$ (the factor of $2$ comes from the fact that in order to change the edges incident with the $i$th vertex, we first remove and then add edges). By taking $c_i=d_i$ and $\Gamma=\prod_i \Lambda_i$ in Lemma~\ref{azuma}, we conclude that
\[\mathbb{P}[|\log r(G_{n,p}) - \mathbb{E}(\log r(G_{n,p}))| > t] \leq 2 \exp\left(-\frac{t^2}{32 n \log^2 n}\right).\]
Taking $t = \omega(\sqrt{n} \log n)$ then implies the result.
\end{proofof}

Theorem~\ref{thm:ramcon} is not entirely satisfying, since it seems very likely that the $\log$ factor is unnecessary. Such a statement would immediately follow if we could establish the following feasible conjecture.

\begin{conjecture} \label{conj:addone}
There exists an absolute constant $c$ such that for any graph $H$ and any graph $H'$ obtained from $H$ by removing a single vertex,
\[r(H) \leq c \cdot r(H').\]
\end{conjecture}

For dense $H$, we can prove this conjecture using the following lemma of Erd\H{o}s and Szemer\'edi~\cite{ESz72}. Recall that the \emph{density} of a graph $G$ with $e(G)$ edges is $e(G)/\binom{v(G)}{2}$, the proportion of possible edges which are actually in the graph.

\begin{lemma} [Erd\H{o}s--Szemer\'edi~\cite{ESz72}] \label{lem:ESz}
There exists an absolute constant $a$ such that the following holds: any complete graph whose edges can be two-colored in red and blue so that there is no monochromatic copy of $K_n$ and the density of blue edges is at most $q \leq 1/2$ has at most $2^{a q \log(1/q) n}$ vertices.
\end{lemma}

\begin{lemma} \label{lem:ESzapp}
There exists an absolute constant $c$ such that for any graph $H$ of density at least $d$ and any graph $H'$ obtained from $H$ by removing a single vertex,
\[r(H) \leq c \frac{\log(1/d)}{d} r(H').\]
\end{lemma}

\begin{proof}
Let $n$ be the number of vertices in $H$ and let $c$ be a sufficiently large constant which we choose later. We first note that a simple application of the probabilistic method implies that $r(H') \geq 2^{c' dn}$ for some fixed positive constant $c'$. Indeed, the number of edges in $H'$ is at least $m'=d{n \choose 2}-(n-1)$. Thus, the probability that a random red/blue edge coloring of the complete graph on $N=2^{c' dn}$ vertices contains a monochromatic copy of $H'$ is at most $2 N^{n-1} 2^{-m'} <1$ for an appropriate choice of $c'$ (e.g., $c'=1/3$ will do). Suppose now that we have a graph with $c \frac{\log(1/d)}{d} r(H')$ vertices containing no monochromatic copy of $H$. Fix a vertex $v$. By the pigeonhole principle, $v$ has at least $\frac{c}{2} \frac{\log(1/d)}{d} r(H')$ neighbors in one color, say red. Call this set of neighbors $X$. If there is a vertex $w$ with at least $r(H')$ blue neighbors in $X$, then this neighborhood contains a monochromatic copy of $H'$, which can be extended to a monochromatic copy of $H$ by adding either $v$ or $w$. Otherwise, the set $X$ has blue density at most $q=\frac{2}{c} \frac{d}{\log(1/d)}$. By choosing $c$ sufficiently large, we can guarantee that $a q \log(1/q) < c' d$. Since the size of $X$ is at least $r(H')\geq 2^{c' dn}$, an application of Lemma~\ref{lem:ESz} implies that $X$  contains a monochromatic copy of $K_n$ and, hence, of $H$.  
\end{proof}

In particular, we have that $r(K_{n+1}) \leq c \cdot r(K_n)$ for some absolute constant $c$, which seems not to have been remarked before, though there is work of Burr, Erd\H{o}s, Faudree and Schelp~\cite{BEFS89} in the opposite direction, giving lower bounds on the differences between consecutive Ramsey numbers. More on point, we may use Lemma~\ref{lem:ESzapp} to improve Theorem~\ref{thm:ramcon} when $p$ is a fixed positive constant.

\begin{theorem} \label{thm:ramcondense}
For $p$ a fixed positive constant,
\[\mathbb{P}[|\log r(G_{n,p}) - \mathbb{E}(\log r(G_{n,p}))| = \omega(\sqrt{n})] = o(1).\]
\end{theorem}

\begin{proof}
We follow the proof of Theorem~\ref{thm:ramcon}. Let $X_1, \dots, X_n$ be the same random variables defined there, encoding the edges of the random graph $G_{n,p}$. Let $\Gamma$ be the event that the density of $G_{n,p}$ is at least $p/2$. By the standard Chernoff estimates for the binomial random variable, the probability that $\Gamma$ does not hold is at most 
$e^{-\Omega(pn^2)}$. As before,  we have that $|f(G) - f(G')| \leq 4 \log n=d_i$ for any two graphs $G$, $G'$
which differ only in the edges incident with the $i$th vertex. On the other hand, since $p$ is a constant, we can use Lemma~\ref{lem:ESzapp} to conclude that if $G \in \Gamma$ then $|f(G) - f(G')| \leq O(1)=c_i$ for any two $G$, $G'$ as above. Define $\gamma_i=1/n$ and note that $\sum_i(c_i+\gamma_i(d_i-c_i))^2=O(n)$. 
Therefore, taking $t = \omega(\sqrt{n})$ and using Lemma \ref{azuma}, we have that
$$\mathbb{P}[|\log r(G_{n,p}) - \mathbb{E}(\log r(G_{n,p}))| > t] \leq 2e^{-\Omega(t^2/n)}+2n^2e^{-\Omega(pn^2)}=
o(1),$$ completing the proof.	
\end{proof}

One might also ask whether results like Theorems~\ref{thm:ramcon} and~\ref{thm:ramcondense} hold for variants of the Ramsey number. For many natural variants, including the induced Ramsey number and the size Ramsey number, we were unable to obtain such statements, chiefly because we could not establish an analogue of Lemma~\ref{lem:bounddiff}. Any progress on such questions would be interesting. 

In the other direction, one might ask whether the random variable $\log r(G_{n,p})$ is not too concentrated. We make the following conjecture along these lines.

\begin{conjecture}\label{conjanti}
For any fixed $0<p<1$, $\log r(G_{n,p})$ is not concentrated with high probability on any interval of length $O(1)$. 
\end{conjecture}

An argument of Alon and Krivelevich~\cite[Section 4]{AK97} shows that this conjecture is true on average in some appropriate sense. First note that for any $c, \epsilon > 0$, there is $\delta > 0$ such that if $\epsilon \leq p \leq 1 - \epsilon$, then any family of graphs on $n$ vertices whose total probability in $G(n,p)$ is at least $1 - \delta$ has probability at least $2\delta$ in $G(n, p + c/n)$. Defining $I_{n, p}$ to be an interval such that $\log r(G_{n,p})$ lies in this interval with probability at least $1 - \delta$, we conclude that $I_{n, p}$ overlaps with $I_{n, p + c/n}$ whenever $\epsilon \leq p \leq 1-\epsilon$. It follows from the main results in \cite{C13} and \cite{S11} that there is a constant $\epsilon>0$ such that, for $p \leq \epsilon$, $\log r(G_{n,p})<n/8$ with high probability. On the other hand, if $p \geq 1-\epsilon$ for $\epsilon>0$ fixed and sufficiently small, a simple application of the probabilistic method (see the proof of Lemma \ref{lem:ESzapp}) implies that $\log r(G_{n,p})>n/4$ with high probability. Therefore, since the midpoints of the intervals $I_{n,\epsilon}$ and $I_{n,1- \epsilon}$ differ by $\Omega(n)$, we must have that the average width of the at most $n/c$ intervals $I_{n, \epsilon}, I_{n, \epsilon + c/n}, I_{n, \epsilon + 2c/n}, \dots$ is $\Omega(c)$. A more careful rendering of this argument gives the following result.

\begin{proposition}\label{easy}
For each $0<p<1$ and positive integer $n$, let $\ell_{n,p}$ be such that $\log r(G_{n,p})$ is with high probability concentrated in an interval of length $\ell_{n,p}$. There is a constant $0<\epsilon<1/2$ such that if $\epsilon \leq p \leq 1-\epsilon$ is taken uniformly at random, then the expected value of $\ell_{n,p}$ is $\omega(1)$. 
\end{proposition}

We conclude with one further question, concerning $r(G_{n, d/n})$ with $d$ fixed, where the Ramsey number is known~\cite{FS092} to be linear in $n$ with high probability. It is not clear to us which way the truth should lie.

\begin{question}
For fixed $d$, is it the case that
\[|r(G_{n, d/n}) - \mathbb{E}(r(G_{n,d/n}))| = o(n)\]
with high probability?
\end{question}

\section{Ramsey multiplicity and the number of colors} \label{sec:mult}

The {\it Ramsey multiplicity} $M_q(H;n)$ of a graph $H$ is defined to be the minimum number of monochromatic copies of $H$ taken over all $q$-colorings of the edges of the complete graph $K_n$. If $H$ has $a$ automorphisms and $h$ vertices, the {\it Ramsey multiplicity constant} $C_q(H)$ is given by 
\[C_q(H) = \lim_{n \rightarrow \infty} \frac{M_q(H;n)}{\frac{h!}{a} \binom{n}{h}},\]
that is, the limit as $n$ tends to infinity of the minimum proportion of copies of $H$ which are monochromatic in any $q$-coloring of the edges of $K_n$.

Writing $m$ for the number of edges in $H$, a simple upper bound, $C_q(H) \leq q^{1 - m}$, follows from considering a random coloring. When $q = 2$ and $H = K_3$, a classical result of Goodman~\cite{G59} shows that this upper bound is tight. This led Erd\H{o}s~\cite{Er62a}, for cliques, and Burr and Rosta~\cite{BR80}, for general $H$, to conjecture that the upper bound is tight in the $2$-color case. However, in the late eighties, Thomason~\cite{T89} showed that these conjectures fail already for $K_4$. 

A strong quantitative counterexample to the Burr--Rosta conjecture was found by the second author~\cite{F08}. Define the graph $T(k, \ell)$ to be a clique of order $k$ one of whose vertices is joined to $\ell$ otherwise isolated vertices. Color the complete graph on $n$ vertices by splitting its vertex set into $k-1$ pieces of equal size and coloring an edge red if it lies entirely within one of these pieces and blue otherwise. This coloring contains no blue copies of $T(k,\ell)$ and any red copy is contained within one of the $k-1$ clumps. Taking $\ell = (k^2+k)/2$, the number of vertices is $h = (k^2 + 3k)/2$ and the number of edges is $m = k^2$, so taking the limit as $n$ tends to infinity implies that 
\[C_2(T(k,\ell)) \leq (k-1)^{1 - h} \leq m^{-(1+o(1))m/4}.\]  

For three or more colors, this counterexample becomes even more effective. Indeed, take the same graph $T(k,\ell)$ with $\ell = (k^2+k)/2$ and suppose that we have a $(q-1)$-coloring of the complete graph on $r = 2^{(q-2)k/4}$ vertices with no monochromatic copy of $K_k$ (see~\cite[Section 2.1]{CFS15} for the proof that such a coloring exists). Consider the blow-up of this coloring, where each vertex is replaced by a clique of color $q$ with size $n/r$. By construction, the only monochromatic cliques of size $k$ have color $q$ and, therefore, the only monochromatic copies of $T(k, \ell)$ are also in color $q$. Letting $n$ tend to infinity, we see that
\[C_q(T(k,\ell)) \leq r^{1 - h} \leq 2^{-c'_q m^{3/2}}.\]

These upper bounds for $C_2(T(k,\ell))$ and $C_q(T(k, \ell))$ with $q \geq 3$ are both tight up to a constant factor in the exponent. This can be proved using an appropriate generalization of Lemma \ref{triv1} below, though we omit the details for brevity. In particular, we see a substantive difference in behavior between the $2$- and $3$-color cases. An obvious objection might be to say that the graph for which this difference happens is an unusual one. In response to this objection, we prove the following theorem, improving an earlier result, $C_2(H) \geq 2^{-c m^{3/2} \log m}$, of the second author \cite{F08}.  

\begin{theorem} \label{TwoColour}
There exists a constant $c$ such that, for any graph $H$ with $m$ edges,
\[C_2(H) \geq 2^{-c m^{4/3} \log^2 m}.\]
\end{theorem}

So the difference in behavior is a very real one: there is no graph on $m$ edges which has a $2$-color multiplicity constant as small as $2^{-c m^{3/2}}$. Along with~\cite{CFR17}, this is one of the first results where there is a provable quantitative difference between the $2$-color case and the $q$-color case for some $q \geq 3$.

To prove Theorem~\ref{TwoColour}, we use a mixture of two methods, both of which have become mainstays of graph Ramsey theory. The first, known as dependent random choice (see the survey~\cite{FS11}), shows that every dense graph contains a large set of vertices $A$ with the property that almost all small subsets of $A$ have many common neighbors. The following lemma, a routine variant of \cite[Lemma 2.1]{FS09} (see also~\cite[Lemma 1]{C09}), 
will suffice for our purposes. Note that we write $N(S)$ to denote the joint neighborhood of a set $S$, the collection of vertices joined to every vertex in $S$.

\begin{lemma} \label{DRC}
Let $G = (U, V; E)$ be a bipartite graph with at least $\epsilon |U||V|$ edges. Then, for all positive integers $a, t, x$, there is a subset $A \subset U$ with $|A| \geq 2^{-1/a} \epsilon^t |U|$ such that all $a$-sets $S$ in $A$ but at most $2 \epsilon^{-ta} \left(\frac{x}{|V|}\right)^t |A|^a/a!$  have $|N(S)| \geq x$.
\end{lemma}

It will also be useful to record the following result.

\begin{lemma}\label{triv1}
Let $k$ and $n$ be positive integers with $n$ sufficiently large. For every red/blue edge coloring of $K_n$, there is a vertex subset $U$ with $|U| \geq n/3$ vertices and a color such that 
\begin{itemize} 
\item there are at least $12^{-k}{n \choose k}$ monochromatic $K_k$ in $U$ in that color and every such monochromatic $K_k$ extends to at least $n/2$ monochromatic $K_{k+1}$, or 
\item every vertex in $U$ has both red degree and blue degree at least $D:=(n-3)/12k$. 
\end{itemize}
\end{lemma}
\begin{proof}
Each vertex has degree at least $(n-1)/2 \geq D$ in some color. So at least one of the following holds: there are at least $n/3$ vertices with both red degree and blue degree at least $D$ (in which case we are done); there are at least $n/3$ vertices of red degree at least $D$ and blue degree less than $D$; or there are at least $n/3$ vertices of blue degree at least $D$ and red degree less than $D$. 

Without loss of generality, we may assume that there are at least $n/3$ vertices of red degree at least $D$ and blue degree less than $D$. We let $U$ be the set of such vertices. The density of blue edges in $U$ is at most $\frac{D|U|/2}{{|U| \choose 2}} \leq \frac{1}{4k}$. Consider a random set $T$ of $2k$ vertices in $U$. The expected number of blue edges in $T$ is thus at most $\frac{1}{4k}{2k \choose 2} \leq \frac{k}{2}$. Thus, with probability at least $1/2$, the number of blue edges in $T$ is at most $k$ and we can delete one vertex from each blue edge in $T$ so that $T$ contains a red $K_k$. It follows that the number of red $K_k$ in $U$ is at least 
$\frac{1}{2}{|U| \choose 2k}/{|U|-k \choose k}=\left(2{2k \choose k}\right)^{-1} {|U| \choose k}$. The lower bound on $|U|$ and the fact that $n$ is sufficiently large implies the desired lower bound on the number of monochromatic red $K_k$ in $U$. Each such red $K_k$ extends to at least $n-k-kD \geq n/2$ red $K_{k+1}$. 
\end{proof}

For our application of dependent random choice, we require some further notation. Recall that a \emph{hypergraph} $\mathcal{F} = (V, E)$ is a vertex set $V$ together with a collection of subsets $E$ of $V$, known as the \emph{edges}. A hypergraph is said to be \emph{$k$-uniform} if all of the edges have the same size $k$. A hypergraph $\mathcal{F} = (V, E)$ is said to be \emph{down-closed} if $e_1 \subset e_2$ and $e_2 \in E$ implies $e_1 \in E$.

\begin{lemma} \label{DRCApplied}
Suppose that the edges of $K_n$ are two-colored in red and blue. Then, provided $n$ is sufficiently large, there are at least $2^{-50 m^{4/3} \log m} \binom{n}{m^{2/3}}$ monochromatic copies of $K_{m^{2/3}}$ such that, for each such copy $K$, every vertex subset $L$ of $K$ of size $2 m^{1/3}$ has at least $2^{-50 m^{1/3} \log m} n$ vertices connected to each element of $L$ in the same color as the clique.
\end{lemma}

\begin{proof}
Let $k=m^{2/3}$ and apply Lemma \ref{triv1} to obtain the promised set $U$ of at least $n/3$ vertices. In the first case, the monochromatic $K_k$ in $U$ give the desired set of monochromatic $K_k$. So we may assume that every vertex in $U$ has both red degree and blue degree at least $D=(n-3)/12k$.

Let $\epsilon =1/13k$, $a = 2 m^{1/3}$, $t = m^{2/3}$ and $x = 2^{-50 m^{1/3} \log m} n$. 
We now apply Lemma \ref{DRC} to the bipartite graph with parts $U$ and $V(K_n)$, where $(u,v)$ is an edge if it is red in the given edge coloring, with the parameters above to find a set $A \subset U$ of size at least $2^{-10 m^{2/3} \log m} n$ such that all but $2^{-30 m \log m}|A|^a/a!$ of the $a$-tuples have at least $2^{-50 m^{1/3} \log m} n$ common red neighbors.

Now, since $A \subset U$, every vertex in $A$ has blue degree at least $D$. Therefore, applying Lemma \ref{DRC} with the same parameters as before we find a set $A' \subset A$ of size at least $2^{-10 m^{2/3} \log m} |A| \geq 2^{-20 m^{2/3} \log m} n$ such that all but $2^{-30 m \log m}|A'|^a/a!$ of the $a$-tuples have at least $2^{-50 m^{1/3} \log m} n$ common blue neighbors. In summary, since $|A| \leq 2^{10 m^{2/3} \log m} |A'|$ and $a=2 m^{1/3}$, we have a set $A'$ of size at least $2^{-20 m^{2/3} \log m} n$ such that all $a$-tuples in $A'$ but at most 
\[2^{-30 m \log m}|A|^a/a! + 2^{-30 m \log m}|A'|^a/a! \leq 2^{-8 m \log m}|A'|^a/a!\] 
have $2^{-50 m^{1/3} \log m} n$ common red neighbors and $2^{-50 m^{1/3} \log m} n$ common blue neighbors.

Let $\mathcal{F}$ be the down-closed hypergraph with vertex set $A'$ whose maximal edges consist of the at least $\left(1 - 2^{-8 m \log m}\right) |A'|^a/a!$ sets of size $a$ which have appropriately large common neighborhoods in both red and blue. We fix $d = 2^{2 m^{2/3}} \binom{2 m^{2/3}}{2m^{1/3}}$. Call a set $S \subset A'$ of size $|S| \leq a$ good if $S$ is contained in more than $(1 - (4d)^{|S| - a})|A'|^{a-|S|}/(a-|S|)!$ edges of $\mathcal{F}$ of cardinality $a$. For any good set $S$ with $|S| < a$ and any vertex $j \in A' \setminus S$, call $j$ bad with respect to $S$ if $S \cup \{j\}$ is not good. Let $B_S$ denote the set of vertices $j \in A'\setminus S$ that are bad with respect to $S$. We claim that if $S$ is good with $|S| < a$, then $B_S \leq |A'|/4d$. Indeed, suppose that $|B_S| > |A'|/4d$. Then the number of $a$-sets containing $S$ that are not edges of $\mathcal{F}$ is at least
\[\frac{|B_S|}{a - |S|} (4d)^{|S| + 1 - a}\frac{|A'|^{a - |S| - 1}}{(a - |S| - 1)!} > (4d)^{|S|-a}\frac{|A'|^{a - |S|}}{(a - |S|)!},\]
which would be a contradiction.

We will now prove that the hypergraph $\mathcal{F}$ contains many copies of the complete $2m^{1/3}$-uniform hypergraph with $m^{2/3}$ vertices such that the graph induced by the vertices of each copy is monochromatic. We will embed $2m^{2/3}$ vertices, vertex by vertex, maintaining three conditions. More specifically, we will inductively find an embedding $f$ of $v_1, v_2, \cdots, v_i$ and a disjoint vertex subset $V_i$ satisfying the following conditions. The first condition is that if we have embedded $i$ vertices $v_1, v_2, \cdots, v_i$, then every subset of these vertices of size less than or equal to $2m^{1/3}$ is good. The second condition is that for every $j, k$ with $1 \leq j < k \leq i$, the color of $f(v_j) f(v_k)$ is determined by the value of $j$. Moreover, the color of the edge from $f(v_j)$ to any element of $V_i$ is the same color. Thirdly, we will maintain that $|V_i| \geq 2^{-i} |A'|$.

To verify the induction hypothesis at $i = 0$, we only have to verify that all empty sets are good, since the second and third conditions follow trivially by taking $V_0 = A'$. However, the empty set is good, since
\[\left(1 - 2^{-8 m \log m}\right)|A'|^a/a! \geq \left(1 - (4d)^{-a}\right) |A'|^a/a!.\]
We may therefore assume that $v_1, v_2, \cdots, v_i$ have been embedded as required and we are trying to embed $v_{i+1}$.

We may suppose, without loss of generality, that there are $|V_i|/2$ vertices in $V_i$ with red degree at least $|V_i|/2$. Moreover, there are at most $\binom{2 m^{2/3}}{2 m^{1/3}}$ subsets of $v_1, v_2, \cdots, v_i$ of size at most $2 m^{1/3}-1$. By induction, all of these sets have a good embedding. For any such subset $S$, there are, therefore, at most $|A'|/4d$ bad vertices with respect to it. But now, since $d = 2^{2 m^{2/3}} \binom{2 m^{2/3}}{2 m^{1/3}}$,
\[\frac{|A'|}{4d} \leq 2^{2 m^{2/3}} \frac{|V_i|}{4d} \leq \frac{|V_i|}{4 \binom{2 m^{2/3}}{2 m^{1/3}}}.\]
Thus, adding the numbers of bad vertices with respect to all the sets to which we would like to append $v_{i+1}$, we obtain that there are at most $|V_i|/4$ vertices which are bad for one of these sets. We may, therefore, let $f(v_{i+1})$ be any of the remaining $\frac{|V_i|}{4}$ vertices which have $|V_i|/2$ red neighbors in $V_i$ and are good for each of the sets $S$. Letting $V_{i+1}$ be the red neighborhood in $V_i$ of $f(v_{i+1})$ completes the induction.

Since we continue this process until we have $2 m^{2/3}$ vertices, some subset $w_1, w_2, \cdots, w_{m^{2/3}}$ must form a monochromatic clique. How many monochromatic cliques do we form through this process? At the very worst,
\begin{align*}
\frac{1}{(2m^{2/3})! n^{m^{2/3}}} \frac{|V_0|}{4} \times \frac{|V_1|}{4} \times \cdots \times \frac{|V_{2m^{2/3}-2}|}{4} \times \frac{|V_{2m^{2/3}-1}|}{4} & \geq 2^{-6m^{2/3} \log m} n^{-m^{2/3}} \left(\frac{|A'|}{2^{2m^{2/3}}}\right)^{2m^{2/3}}\\ 
& \geq 2^{-50 m^{4/3} \log m} \binom{n}{m^{2/3}},
\end{align*}
as required. This completes the proof.
\end{proof}

The second technique we use goes back to work of Graham, R\"{o}dl and Ruci\'nski~\cite{GRR00}, though similar ideas are implicit in earlier work~\cite{KPR98}. Let $G$ be a graph on vertex set $V$ and let $X, Y$ be two subsets of $V$. Define $e(X,Y)$ to be the number of edges between $X$ and $Y$. The \emph{density} of the pair $(X, Y)$ is $d(X,Y) = e(X,Y)/|X||Y|$.
The graph $G$ is said to be {\it bi-$(\sigma, \delta)$-dense} if, for all $X, Y \subset V$ with $X \cap Y = \emptyset$ and $|X|, |Y| \geq \sigma |V|$, we have $d(X,Y) \geq \delta$. It was shown by Graham, R\"{o}dl and Ruci\'nski \cite[Lemma 2]{GRR00} that if $\sigma$ is sufficiently small in terms of $\delta$ and the maximum degree $\Delta$ of a fixed graph $H$, then a sufficiently large bi-$(\sigma, \delta)$-dense graph $G$ will contain a copy of $H$. The following lemma is a counting version of their result with restrictions on where each vertex can be embedded.

\begin{lemma} \label{Embedding}
Let $\alpha, \delta > 0$ be real numbers. If $G$ is a bi-$(\alpha \delta^{\Delta}/2t^2, \delta)$-dense graph with $n$ vertices, then $G$ contains at least $\left(\alpha \delta^\Delta n/2t^2\right)^{t}$ copies of any graph $H$ with $t$ vertices and maximum degree $\Delta$ such that each vertex $w$ of $H$ is mapped to a prespecified vertex subset $V_w$ of size at least $\alpha n$.
\end{lemma}

\begin{proof}
Our initial aim will be to find a single embedding $f$ of $H$ in $G$ so that $f(w) \subset V_w$ for all $w \in V(H)$. By replacing each $V_w$ with a set of order $|V_w|/t$, we may assume that the $V_w$ are disjoint. Let the vertices of $H$ be $\{w_1, w_2, \cdots, w_t\}$. For each $1 \leq h \leq t$, let $L_h = \{w_1, w_2, \cdots, w_h\}$. If the vertices of $L_h$ have been embedded, then, for each $y \in V(H) \setminus L_h$, we let $T_y^h$ be the set of vertices in $V_y$ which are adjacent to all already embedded neighbors of $y$. That is, letting $N_h (y) = N(y) \cap L_h$, $T_y^h$ is the set of vertices in $V_y$ adjacent to each element of $f(N_h (y))$. We will find, by induction, an embedding of $L_h$ such that, for each $y \in V(H) \setminus L_h$, $|T_y^h| \geq \delta^{|N_h(y)|} |V_y|$. 

For $h = 0$, there is nothing to prove. We may therefore assume that $L_h$ has been embedded consistent with the induction hypothesis and attempt to embed $w = w_{h+1}$ into an appropriate $v \in T_w^h$. Let $Y$ be the set of neighbors of $w$ which are not yet embedded. We wish to find an element $v \in T_w^h$ such that, for all $y \in Y$, $|N(v) \cap T_y^h| \geq \delta |T_y^h|$. If such a vertex $v$ exists, taking $f(w) = v$ will then complete the proof.

Let $B_y$ be the set of vertices in $T_w^h$ which are bad for $y \in Y$, that is, such that $|N(v) \cap T_y^h| < \delta |T_y^h|$. Note that, by the induction hypothesis, $|T_y^h| \geq \delta^{\Delta} |V_y| \geq \alpha \delta^\Delta n/t$. Therefore, $|B_y| < \alpha \delta^{\Delta} n/2t^2$, for otherwise the density between the sets $B_y$ and $T_y^h$ would be less than $\delta$, contradicting the bi-density condition. Hence,
\[\left| T_w^h \setminus \cup_{y \in Y} B_y \right| > \frac{\alpha \delta^\Delta}{t} n - t \frac{\alpha \delta^\Delta}{2t^2} n \geq \frac{\alpha \delta^\Delta}{2t} n\]
and an appropriate choice for $f(w)$ exists. The result therefore follows by induction. The counting result also follows straightforwardly since there are at least $\frac{\alpha \delta^\Delta}{2t} n$ choices for each of the $t$ vertices and each copy of $H$ can be counted at most $t! \leq t^t$ times.
\end{proof}

Thus, if we are looking for monochromatic copies of $H$ in a red/blue edge coloring of $K_n$ and we cannot find a large number of red copies, there must be a large bipartite graph where the blue density is very high. If we iterate this observation, we might hope to find a sequence of such bipartite graphs, enough to build up a complete graph with very high blue density. Precisely this idea was pursued by Graham, R\"{o}dl and Ruci\'nski, though we use a streamlined approach due to Fox and Sudakov \cite{FS08}. We start with some notation and a definition.

For a graph $G = (V, E)$ and disjoint subsets $W_1, \cdots, W_t \subset V$, the \emph{density} $d_G(W_1, \cdots, W_t)$ between the $t \geq 2$ vertex subsets $W_1, \cdots, W_t$ is defined by 
\[d_G(W_1, \cdots, W_t) = \frac{\sum_{i < j} e(W_i, W_j)}{\sum_{i<j} |W_i||W_j|}.\]

\begin{definition}
For $\alpha, \sigma, \delta \in [0,1]$ and positive integer $t$, a sequence $(G_1, \cdots, G_r)$ of graphs on the same vertex set $V$ is $(\alpha, \sigma, \delta, t)$-sparse if for all subsets $U \subset V$ with $|U| \geq \alpha |V|$, there are positive integers $t_1, \cdots, t_r$ such that $\prod_{i=1}^r t_i \geq t$ and for each $i \in [r] = \{1, 2, \cdots, r\}$, there are disjoint subsets $W_{i,1}, \cdots, W_{i,t_i} \subset U$ with $|W_{i,1}| = \cdots = |W_{i,t_i}| = \lceil \sigma |U| \rceil$ and $d_{G_i}(W_{i,1}, \cdots, W_{i,t_i}) \leq \delta$.
\end{definition}

The fundamental lemma that we will need to apply is the following~\cite[Corollary 3.4]{FS08}.

\begin{lemma} \label{SpBipImpliesSp}
If $(G_1, \cdots, G_r)$ is $\left((\frac{\sigma}{2})^{h-1}, \sigma, \frac{\delta}{8}, 2\right)$-sparse where $h = r \log \frac{2}{\delta}$, then there is $i \in [r]$ and an induced subgraph $G'$ of $G_i$ on $2 \delta^{-1} 2^{1-h} \sigma^h |V|$ vertices that has edge density at most $\delta$.
\end{lemma}

We only use Lemma \ref{SpBipImpliesSp} in the case $r=2$ with $G_1$ and $G_2$ the subgraphs of each color in a  two-coloring of the edges of $K_n$. We are now ready to complete the proof of Theorem \ref{TwoColour}. 

\begin{proofof}{Theorem~\ref{TwoColour}}
Let $\delta = 1/m^3$, $\Delta = 2 m^{1/3}$ and $\sigma =  2^{-60 m^{1/3} \log m}$. Then $h = 2 \log (2m^3)$. Note that
\[\left(\frac{\sigma}{2}\right)^{h-1} \geq 2^{-600 m^{1/3} \log^2 m}.\]
Let us refer to this latter quantity as $\beta$. Let $U$ be a subset of $V$ of size at least $\beta n$. We will first show that, for any graph $H$ with $m$ edges and no isolated vertices, $U$ contains either $2^{-3000 m^{4/3} \log^2 m} \binom{n}{|V(H)|}$ monochromatic copies of $H$ or two disjoint sets $W_1$ and $W_2$ of size $\sigma |U|$ such that the density of edges between them in one color is less than $\delta/8$.

If $H$ has at most $m^{2/3}$ vertices, then $H$ is a subgraph of $K_{m^{2/3}}$, so $C_2(H) \geq C_2(K_{m^{2/3}}) \geq 2.19^{-m^{4/3}}$ by a result of the first author \cite{C209}. We may therefore assume that $H$ has more than $m^{2/3}$ vertices. We may also assume that $H$ has no isolated vertices, since these do not contribute to the multiplicity constant, and hence that $|V(H)| \leq 2m$.

We partition the vertex set of $H$ into two subsets, those of large degree and those of small degree. More specifically, let $A_1$ be the $m^{2/3}$ vertices with largest degree and let $A_2$ be the complement. By choice, every vertex in $A_2$ has maximum degree at most $2 m^{1/3}$. Otherwise, every vertex in $A_1$ would have degree greater than $2 m^{1/3}$, implying the existence of more than $m$ edges.

Applying Lemma \ref{DRCApplied}, we find at least $2^{-50 m^{4/3} \log m} \binom{|U|}{m^{2/3}} \geq 2^{-700 m^{4/3} \log m} \binom{n}{m^{2/3}}$ monochromatic copies of $K_{m^{2/3}}$ such that, for each such copy $K$, every vertex subset $L$ of $K$ of size $2m^{1/3}$ has at least $2^{-50 m^{1/3} \log m} |U|$ vertices connected to each element of $L$ in the same color as the clique. Suppose, without loss of generality, that at least half of these cliques are blue.

Every vertex in $A_2$ has degree at most $2 m^{1/3}$ in $A_1$, so, given any blue $K_{m^{2/3}}$, considered as an embedding of $A_1$, there is, for each $w \in A_2$, a subset of size at most $2 m^{1/3}$ that the image of $w$ should be joined to in blue.
We know that each such vertex set has at least $s = 2^{-50 m^{1/3} \log m} |U|$ common blue neighbors. For any vertex $w \in A_2$, let this set of common neighbors be $V_w$. We would now like to embed the induced subgraph of $H$ on $A_2$ in such a way that $w$ is embedded in $V_w$ for each $w \in A_2$. Applying Lemma~\ref{Embedding} with $\alpha = 2^{-50 m^{1/3} \log m}$ and $\delta_0 = \delta/8$, we see that if the blue graph induced on $U$ is bi-$(\alpha \delta_0^{\Delta}/8m^2, \delta_0)$-dense, then, since $|V_w| \geq \alpha |U|$  for each $w \in A_2$ and $\Delta \leq 2m^{1/3}$, there are at least 
\[\left(\frac{\alpha \delta_0^{\Delta}}{8m^2} |U|\right)^{|A_2|} \geq (2^{-60 m^{1/3} \log m} |U|)^{|A_2|} \geq \left(2^{-700 m^{1/3} \log^2 m} n\right)^{|A_2|}\]
blue copies of $H[A_2]$ which form a blue copy of $H$ with our given copy of $H[A_1]$. Since $|A_2| \leq |V(H)| \leq 2m$, we therefore see that $K_n$ contains at least
\[\frac{1}{|V(H)|!}2^{-701 m^{4/3} \log m} \binom{n}{m^{2/3}} \times \left(2^{-700 m^{1/3} \log^2 m} n\right)^{|A_2|} \geq 2^{-3000 m^{4/3} \log^2 m} \binom{n}{|V(H)|}\]
blue copies of $H$, as required.

If, on the other hand, every subset of order $\beta n$ is not bi-$(\alpha \delta_0^{\Delta}/8m^2, \delta_0)$-dense either for red or blue, then, since $\sigma \leq \alpha \delta_0^{\Delta}/8m^2$ and $\left(\frac{\sigma}{2}\right)^{h-1}\geq \beta$, $(G_B, G_R)$ is $(\left(\frac{\sigma}{2}\right)^{h-1}, \sigma, \frac{\delta}{8}, 2)$-sparse, where $G_B$ and $G_R$ represent the blue and red graphs, respectively. Therefore, by Lemma \ref{SpBipImpliesSp}, for one of the colors, blue say, there is a set $|W|$ of size $2 \delta^{-1} 2^{1-h} \sigma^h n \geq 2^{-600 m^{1/3} \log^2 m} n$ such that the induced graph $G_B[W]$ has density less than $\delta=\frac{1}{m^3}$. Therefore, a random subset of $|V(H)| \leq 2m$ vertices from $W$ has in expectation at most ${|V(H)| \choose 2}/m^3 \leq 1/2$ blue edges and so is monochromatic red with probability at least $1/2$. Hence, the number of red copies of $K_{|V(H)|}$ in $W$ is at least $$\frac{1}{2}{|W| \choose |V(H)|} \geq \frac{1}{3}\frac{|W|^{|V(H)|}}{|V(H)|!} \geq 2^{-1500 m^{4/3} \log^2 m} \binom{n}{|V(H)|}.$$
This completes the proof. 
\end{proofof}

We do not expect Theorem~\ref{TwoColour} to be the end of the story. Indeed, as conjectured in~\cite{F08}, it is likely that $C_2(H) = 2^{-m^{1+o(1)}}$. However, this conjecture seems well beyond the reach of current methods. Similarly, it is tempting to conjecture that for $q \geq 3$, we have $C_q(H) \geq 2^{-m^{3/2 + o(1)}}$.  Here, the state of the art, easily derivable from a sampling argument and the bound $r_q(H) \leq 2^{c_q m^{2/3}}$ for the $q$-color Ramsey number of a graph $H$ with $m$ edges and no isolated vertices~\cite{JP16}, is $C_q(H) \geq 2^{-m^{5/3+o(1)}}$.

\section{A Ramsey problem for connected matchings} \label{sec:match}

In this section, we discuss a common generalization of two classical problems, the Ramsey number of a triangle versus a large clique and Hadwiger's conjecture for graphs with independence number two. The \emph{Ramsey number} $r(3,t)$ is the minimum integer $n$ such that every triangle-free graph on $n$ vertices has independence number at least $t$. The problem of estimating $r(3,t)$ has been studied extensively over the last 60 years. After several successive improvements, the asymptotic behavior of $r(3,t)$ was determined by Ajtai, Komlos and Szemer\'edi \cite{AjKoSz} and by Kim \cite{Kim},
 who proved upper and lower bounds showing that $r(3,t)=\Theta(t^2/\log t)$.

The celebrated Hadwiger conjecture states that any graph $G$ with chromatic number $k$ has a clique minor of order $k$. Since the chromatic number of a graph $G$ is always at least $|V(G)|/\alpha(G)$, this conjecture implies that any graph $G$ on $n$ vertices has a clique minor of order at least $n/\alpha(G)$. The problem of proving that this implication holds when $\alpha(G) = 2$ has attracted considerable attention in recent years. The aim in this case is to show that an $n$-vertex graph with independence number two has a clique minor of order $n/2$. A result of Duchet and Meyniel~\cite{DuMe} demonstrates that such a $G$ has a clique minor of order at least $n/3$. Seymour and Mader have both asked whether the factor $1/3$ can even be improved to $1/3+\epsilon$ for some constant $\epsilon>0$. 

This problem of improving the bound for the maximum clique minor in a graph with independence number two can be reformulated in terms of an interesting Ramsey-type problem. A set of pairwise disjoint edges
$e_1, \ldots, e_t$ of $G$ is called a {\it connected matching of size $t$} if, for every pair of distinct edges $e_i, e_j$, there is at least one edge of $G$ connecting an endpoint of $e_i$ to an endpoint of $e_j$. 
It was observed by Thomass\'e that the following conjecture  is equivalent to the above problem of improving the constant factor $1/3$.

\begin{conjecture} \label{connected-matching}
There exists a constant $C$ such that every $n$-vertex graph $G$ with $\alpha(G)=2$ contains a connected matching of size at least $n/C$. 
\end{conjecture}  

It was proved in \cite{KaPlTo} that if a graph $G$ with independence number two has a connected matching of size $t$, then it has a clique minor of order at least $n/3+t/9$. In the other direction, the same authors showed that if $G$ has a clique minor of order at least $n/3+t$, 
then $G$ contains a connected matching of size at least $\frac{3t}{4} - 1$. Improving on earlier work, Fox~\cite{Fox} showed that an $n$-vertex graph $G$ with $\alpha(G)=2$ has a connected matching of size $\Omega(n^{4/5} \log^{1/5}n)$, which remains the best known bound for Conjecture~\ref{connected-matching}.

Motivated by the above conjecture, F\"uredi, Gy\'arf\'as and Simonyi \cite{FuGySi} proposed a more general question. A set of pairwise disjoint edges
$e_1, \ldots, e_t$ of $G$ is called an {\it $s$-connected matching of size $t$} if, for every pair of distinct edges $e_i, e_j$, there are at least $s$ edges of $G$ connecting an endpoint of $e_i$ to an endpoint of $e_j$. For $1 \leq s \leq 4$, let $f_s(t)$ be the minimum $n$ such that every graph on $n$ vertices with independence number two contains an $s$-connected matching of size $t$. Note that a $4$-connected matching is just a clique of order $2t$, so the function $f_4(t)$ is just the Ramsey number $r(3,2t)$ mentioned above and has order of magnitude $\Theta(t^2/\log t)$. Using this notation, Conjecture \ref{connected-matching} states that $f_1(t)$ is linear in $t$. F\"uredi, Gy\'arf\'as and Simonyi \cite{FuGySi} asked to determine the order of magnitude of $f_s(t)$. As a first step towards this goal, they asked whether the functions can be separated, i.e., whether
$f_1(t) \ll f_2(t) \ll f_3(t) \ll f_4(t)$. They also proved that $f_2(t) \leq O(t^{3/2})$. We improve this bound to $f_2(t) \leq O(t^{4/3}\log^{-1/3} t)$. We also prove that $f_3(t) \leq O(t^{3/2})$, which implies that $f_3(t) \ll f_4(t)$, partially answering their question.

\begin{theorem}
There are positive constants $c, c'$ such that if $G$ is an $n$-vertex graph with $\alpha(G)=2$, then $G$ contains a $3$-connected matching of size 
$c n^{2/3}$ and a $2$-connected matching of size $c' n^{3/4} \log^{1/4} n$.
\end{theorem}

\begin{proof}
Let $t$ be the size of the matching we are trying to find. Since $\alpha(G)=2$, note that for every vertex $v \in G$ its non-neighbors form a clique in $G$. Therefore, we can assume that every vertex is adjacent to all but at most $2t-1$ vertices, i.e., $\delta(G) \geq n-2t$, or we are done.
Pick two vertices $u, v$ of $G$ uniformly at random. Let $A_{u,v}$ be the set of vertices adjacent to neither $u$ nor $v$ and $B_{u,v}$ the set of vertices adjacent to at most one of $u$ or $v$. By the discussion above, it is easy to see that $|B_{u,v}| \leq 2 (n-\delta(G)) \leq 4t$. We need the following claim.

\begin{claim}
With probability at least $3/5$, the set $A_{u,v}$ has size at most $10 t^2/n$.
\end{claim}

\begin{proof}
Note that a vertex $w$ of $G$ belongs to $A_{u,v}$ only if both $u$ and $v$ are non-neighbors of $w$. The probability of this event is at most
$\big(\frac{n-\delta(G)}{n}\big)^2 \leq 4t^2/n^2$. Therefore, the expected size of  $A_{u,v}$ is at most $4t^2/n$. The result now follows from Markov's inequality.
\end{proof}

This claim implies that all but at most $0.4{n \choose 2}\leq n^2/5$ pairs $u,v$ in $G$ have $|A_{u,v}| \leq 10t^2/n$. Since $G$ has at least
$n (n-2t)/2$ edges and $t \ll n$, we conclude that there is a set $F$ of at least $n^2/4$ edges of $G$ such that $|A_{u,v}| \leq 10t^2/n$ for every edge $(u,v) \in F$. 

To find a 3-connected matching in $G$, choose $t=0.2 \,n^{2/3}$ and consider an auxiliary graph $H$ whose vertex set is $F$ and where two vertices $(u, v)$ and $(u', v')$ are adjacent if 
these pairs share a vertex or there are at most two edges of $G$ between $(u, v)$ and $(u', v')$. Note that by definition an independent set in $H$ corresponds to a $3$-connected matching in $G$.
To find a large independent set in $H$, we will estimate its maximum degree. Given a pair $(u, v) \in F$, there are at most $2n$ other pairs which intersect $(u, v)$. Note also that if between disjoint pairs $(u, v)$ and $(u', v')$ there are at most
2 edges, then either both $u', v' \in B_{u,v}$ or at least one of them is in $A_{u,v}$. Therefore, the number of pairs $(u', v')$ incident to  
$(u, v)$ in $H$ is at most 
$$\Delta(H) \leq 2n+ {|B_{u,v}| \choose 2}+|A_{u,v}| \cdot n \leq 2n + 8t^2 +10t^2 \leq 19t^2,$$
where we used that $n \ll t^2$. Thus, a well known lower bound on $\alpha(H)$ implies that it contains an independent set of size at least $\frac{|V(H)|}{\Delta(H)+1} \geq \frac{n^2/4}{20t^2}>t$. 

The proof for 2-connected matchings is similar, but with one additional twist. Let $t=c' n^{3/4} \log^{1/4} n$ for some $c'$ we choose later and let $H'$ be the graph with vertex set $F$ and where two vertices $(u, v)$ and $(u', v')$ are adjacent if 
these pairs share a vertex or there is at most one edge of $G$ between $(u, v)$ and $(u', v')$. An independent set in $H'$ will give us a 2-connected matching. To estimate the maximum degree of $H'$, note that if between disjoint pairs $(u, v)$ and $(u', v')$ there is at most one edge,
then either both $u', v' \in A_{u,v}$ or one of them is in $A_{u, v}$ and the other is in $B_{u,v}$. 
Therefore, the number of pairs $(u', v')$ incident to  
$(u, v)$ in $H'$ is at most 
$$\Delta(H') \leq 2n+ {|A_{u,v}| \choose 2}+|A_{u,v}| \cdot |B_{u,v}| \leq 2n + 50t^4/n^2 +40t^3/n \leq 41 t^3/n,$$
where we used that both $n, t^4/n^2 \ll t^3/n$. Next we observe that the vertex set of $H'$ does not contain three disjoint pairs
$(u_i, v_i), 1 \leq i \leq 3$, which form a triangle. Suppose otherwise. Then, since the number of edges between every pair is at most one, we can assume, without loss of generality, that $u_1$ is a non-neighbor of both $u_2$ and $v_2$. Since also $u_1$ has at most one neighbor in
$u_3, v_3$, we can assume it is not adjacent to $u_3$. Moreover, $u_3$ has at most one neighbor in $u_2, v_2$. But then, if $u_3$ is not adjacent to $v_2$, say, $u_1, v_2, u_3$ form an independent set in $G$, contradicting $\alpha(G)=2$.
Thus, every triangle in $H'$ has at least two vertices whose corresponding pairs intersect. Therefore, the total number of triangles $t(H')$ in $H'$ satisfies
$t(H') \leq 2n \cdot \Delta(H')\cdot |V(H')| \leq 82t^3 |V(H')|$. To conclude the proof, we use a well known bound (\cite{Bol}, Lemma 12.16) on the independence number of a graph with few triangles (see also~\cite{AlKrSu} for a more general result).
This bound says that 
$$\alpha(H') \geq 0.1 \frac{|V(H')|}{\Delta(H')} \Big(\log \Delta(H')-0.5 \log (t(H')/|V(H')|)\Big) \geq \Omega\Big(\frac{n^3}{t^3} \log (t^{3/2}/n)\Big) =\Omega\Big(\frac{n^3}{t^3} \log n \Big),$$
where we used that $|V(H')| \geq n^2/4, \Delta(H')\leq 41 t^3/n$ and $t^{3/2}/n >n^{1/8}$. By choosing $c'$ in the definition of $t$ sufficiently small, we have that $\alpha(H')>t$, completing the proof.
\end{proof}

In order to obtain further separation results between the functions $f_1(t), f_2(t)$ and $f_3(t)$, one needs to obtain good lower bounds. In particular, it would be interesting to show that $f_2(t) \geq t^{1+\epsilon}$ for some constant
$\epsilon>0$.

\vspace{3mm}
\noindent
{\bf Acknowledgements.} Section~\ref{sec:ind} was first written in May 2015, predating a recent paper of Geneson~\cite{G18} showing that $T_k \leq k^{5/2 + o(1)}$. 
More recently J. Balogh, W. Linz and L. Mattos \cite{BLM} independently investigated the question of estimating $T_k$ and showed that $T_k = k^{2 + o(1)}$ (which is slightly weaker than Corollary \ref{T_k}). 
We would like to thank Kevin Ford for some helpful discussions on this theme and also thank the anonymous referees for their careful reading of the paper and useful suggestions.

\end{document}